\documentclass[reqno]{amsart}
\usepackage{latexsym,amsfonts,amssymb,epsfig}
\usepackage{hyperref} 
\DeclareMathOperator{\ch}{char}
\DeclareMathOperator{\ad}{ad}
\DeclareMathOperator{\gr}{gr}
\DeclareMathOperator{\Der}{Der}

\DeclareMathOperator{\Lie}{Lie}
\DeclareMathOperator{\Alg}{Alg}

\DeclareMathOperator{\End}{End}
\DeclareMathOperator{\wt}{wt}    
\DeclareMathOperator{\swt}{swt}  
\DeclareMathOperator{\Wt}{Wt}    
\DeclareMathOperator{\Gr}{Gr}     
\renewcommand {\limsup}{\operatorname* {\overline{lim}}}
\renewcommand {\liminf}{\operatorname* {\underline{lim}}}
\DeclareMathOperator{\GKdim}{GKdim}
\DeclareMathOperator{\LGKdim}{\underline{GKdim}}

\renewcommand{\a}{\alpha}
\renewcommand{\b}{\beta}

\newcommand{\dd}{\partial}

\newcommand{\Z}{\mathbb Z}            
\newcommand{\R}{\mathbb R}            
\newcommand{\N}{\mathbb N}            
\newcommand{\F}{\mathbb F}            
\newcommand{\C}{\mathbb C}            
\renewcommand{\H}{\mathcal H}         

\newcommand{\LL}{\mathbf L}        
\newcommand{\QQ}{\mathbf Q}        
\newcommand{\RR}{\mathbf R}        
\renewcommand{\AA}{\mathbf A}      
\newcommand{\uu}{\mathbf u}         
\newcommand{\WW}{\mathbf W}         

\newtheorem{Theorem}{Theorem}[section]
\newtheorem{Corollary}[Theorem]{Corollary}
\newtheorem{Lemma}[Theorem]{Lemma}
\theoremstyle{Remark}
\newtheorem{Remark}{Remark}
\theoremstyle{Example}

\renewcommand{\theenumi}{\roman{enumi}}   

\begin{document}
\title{Fractal just infinite nil Lie superalgebra of finite width}
\author{Otto Augusto de Morais Costa}
\address{Federal Institute of Rio Grande do Norte, 
59112-490 Natal RN, Brazil}
\email{ottoaugustomorais@gmail.com}
\author{Victor Petrogradsky}
\address{Department of Mathematics, University of Brasilia, 70910-900 Brasilia DF, Brazil}
\email{petrogradsky@rambler.ru}
\thanks{The second author was partially supported by grants CNPq~309542/2016-2, FAPESP~2016/18068-9
}

\begin{abstract}
The Grigorchuk and Gupta-Sidki groups play fundamental role in modern group theory.
They are natural examples of self-similar finitely generated periodic groups.
As their natural analogues, there are constructions of 
nil Lie $p$-algebras
over a field of characteristic $2$~\cite{Pe06} and arbitrary positive characteristic~\cite{ShZe08}.
In characteristic zero, similar examples of Lie algebras   do not exist
by a result of Martinez and Zelmanov~\cite{MaZe99}.
\par
The second author constructed analogues of the Grigorchuk and Gupta-Sidki groups
in the world of Lie superalgebras of arbitrary characteristic,
the virtue of that construction is that the Lie superalgebras have clear monomial bases~\cite{Pe16}.
That Lie superalgebras have slow polynomial growth and are graded by multidegree in the generators.
In particular, a self-similar Lie superalgebra $\mathbf{Q}$ is
$\mathbb{Z}^3$-graded by multidegree in 3 generators, its
$\mathbb{Z}^3$-components lie inside an elliptic paraboloid in space, the components are at most one-dimensional,
thus, the $\mathbb{Z}^3$-grading of $\mathbf{Q}$ is fine.
An analogue of the periodicity is that homogeneous elements of the grading $\mathbf{Q}=\mathbf{Q}_{\bar 0}\oplus\mathbf{Q}_{\bar 1}$ are
$\mathrm{ad}$-nilpotent.
In particular, $\mathbf{Q}$ is a nil finely graded Lie superalgebra, which shows
that an extension of the mentioned result of Martinez and Zelmanov~\cite{MaZe99}
to the Lie superalgebras of characteristic zero is not valid.
But computations with $\mathbf{Q}$ are rather technical.
\par
In this paper, we construct a similar but simpler and "smaller" example.
Namely, we construct a 2-generated fractal
Lie superalgebra $\mathbf{R}$ over arbitrary field.
We find a clear monomial basis of $\mathbf{R}$ and, unlike many examples studied before,
we find also a clear monomial basis of its associative hull $\AA$, the latter
has a quadratic growth.
The algebras $\mathbf{R}$ and $\AA$ are $\mathbb{Z}^2$-graded by multidegree in the generators, positions of their
$\mathbb{Z}^2$-components are bounded by pairs of logarithmic curves on plane.
The $\mathbb{Z}^2$-components of $\mathbf{R}$ are at most one-dimensional, thus, the $\mathbb{Z}^2$-grading of $\mathbf{R}$ is fine.
As an analogue of periodicity,
we establish that homogeneous elements of the grading $\mathbf{R}=\mathbf{R}_{\bar 0}\oplus\mathbf{R}_{\bar 1}$ are $\ad$-nilpotent.
In case of $\mathbb{N}$-graded algebras, a close analogue to being simple is being just infinite.
Unlike previous examples of Lie superalgebras, we are able to prove that $\mathbf{R}$ is just infinite,
but not hereditary just infinite.
Our example is close to the smallest possible example, because $\mathbf{R}$ has a linear growth
with a growth function $\gamma_\mathbf{R}(m)\approx 3m$, as $m\to\infty$.
Moreover, $\RR$ is of finite width 4 ($\ch K\ne 2$).
In case $\ch K=2$, we obtain a Lie algebra of width 2 that is not thin.
\par
Thus, we have got a more handy analogue of the Grigorchuk and Gupta-Sidki groups.
The constructed Lie superalgebra $\mathbf{R}$ is of linear growth, of finite width 4, and just infinite. 
It also shows that an extension of the result of Martinez and Zelmanov~\cite{MaZe99}
to the Lie superalgebras of characteristic zero is not valid.
\end{abstract}
\maketitle
\section{Introduction: Self-similar groups and algebras}

\subsection{Golod-Shafarevich algebras and groups}
The General Burnside Problem puts the question whether a finitely generated periodic group is finite.
The first negative answer was given by Golod and Shafarevich,
who proved that, for each prime $p$, there exists a finitely generated infinite $p$-group~\cite{Golod64,GolSha64}.
The construction is based on a famous construction of a family of finitely generated
infinite dimensional associative nil-algebras~\cite{Golod64}.
This construction also yields examples of infinite dimensional finitely generated Lie algebras $L$
such that $(\ad x)^{n(x,y)}(y)=0$, for all $x,y\in L$, the field being arbitrary~\cite{Golod69}.
The field being of positive characteristic $p$,
one obtains an infinite dimensional finitely generated restricted Lie algebra $L$ such that the $p$-mapping is nil, namely,
$x^{[p^{n(x)}]}=0$, for all $x\in L$.
This gives a negative answer to a question of Jacobson whether
a finitely generated restricted Lie algebra $L$ is finite dimensional provided that
each element $x\in L$ is algebraic, i.e. satisfies some $p$-polynomial $f_{p,x}(x)=0$
(\cite[Ch.~5, ex.~17]{JacLie}). 
It is known that the construction of Golod yields associative nil-algebras of exponential growth.
Using specially chosen relations, Lenagan and Smoktunowicz
constructed associative nil-algebras of polynomial growth~\cite{LenSmo07}.
On further developments concerning Golod-Shafarevich algebras and groups see~\cite{Voden09}, \cite{Ershov12}.

A close by spirit but different construction was motivated by respective group-theoretic results.
A restricted Lie algebra $G$ is called {\it large} if there is a subalgebra  $H\subset G$ of finite codimension
such that $H$ admits a surjective homomorphism on a nonabelian free restricted Lie algebra.
Let $K$ be a perfect at most countable field of positive characteristic.
Then there exist infinite dimensional finitely generated nil restricted Lie algebras over $K$ that
are residually finite dimensional and direct limits of large restricted Lie algebras~\cite{BaOl07}.

\subsection{Grigorchuk and Gupta-Sidki groups}
The construction of Golod is rather undirect, Grigorchuk gave a direct and elegant construction of
an infinite 2-group generated by three elements of order 2~\cite{Grigorchuk80}.
This group was defined as a group of transformations of the interval $[0,1]$ from which
rational points of the form $\{k/2^n\mid  0\le k\le 2^n,\ n\ge 0\}$ are removed.
For each prime $p\ge 3$, Gupta and Sidki gave a direct construction of an infinite $p$-group
on two generators, each of order $p$~\cite{GuptaSidki83}.
This group was constructed as a subgroup of an automorphism group of an infinite regular tree of degree $p$.

The Grigorchuk and Gupta-Sidki groups are counterexamples to the General Burnside Problem.
Moreover, they gave answers to important problems in group theory.
So, the Grigorchuk group and its further generalizations
are first examples of groups of intermediate growth~\cite{Grigorchuk84}, thus answering
in negative to a conjecture of Milnor that groups of intermediate growth do not exist.
The construction of Gupta-Sidki also yields groups of subexponential growth~\cite{FabGup85}.
The Grigorchuk and Gupta-Sidki groups are {\it self-similar}.
Now self-similar, and so called {\it branch groups}, form a well-established area in group theory, see
for further developments~\cite{Grigorchuk00horizons,Nekr05}.
Below we discuss existence of analogues of the Grigorchuk and Gupta-Sidki groups for other algebraic structures.

\subsection{Self-similar nil graded associative algebras}
The study of these groups lead to investigation of group rings and other related associative algebras~\cite{Sidki97}.
In particular, there appeared self-similar associative algebras defined by matrices
in a recurrent way~\cite{Bartholdi06}.
Sidki suggested two examples of self-similar associative matrix algebras~\cite{Sidki09}.
A more general family of self-similar associative algebras was introduced in~\cite{PeSh13ass},
this family generalizes the second example of Sidki~\cite{Sidki09},
also it yields a realization of a Fibonacci restricted Lie algebras (see below)
in terms of self-similar matrices~\cite{PeSh13ass}.
Another important feature of some associative algebras $A$ constructed in~\cite{PeSh13ass} is that
they are sums of two locally nilpotent subalgebras $A=A_+\oplus A_-$
(see similar decompositions~\eqref{decompLL} below).
Recall that an algebra is said {\em locally nilpotent} if every finitely generated subalgebra is nilpotent.
But the desired analogues of the Grigorchuk and Gupta-Sidki groups should be associative self-similar nil-algebras,
in a standard way yielding new examples of finitely generated periodic groups.
But such examples are not known yet.
On similar open problems in theory of infinite dimensional algebras see review~\cite{Zelmanov07}.

\subsection{Self-similar nil restricted Lie algebras, Fibonacci Lie algebra}
Unlike associative algebras, for restricted Lie algebras,
there are known natural analogues of the Grigorchuk and Gupta-Sidki groups.
Namely, over a field of characteristic 2,
the second author constructed an example of an infinite dimensional restricted Lie algebra $\LL$
generated by two elements, called a {\it Fibonacci restricted Lie algebra}~\cite{Pe06}.
Let $\ch K=p=2$ and $R=K[t_i| i\ge 0 ]/(t_i^p| i\ge 0)$ a truncated polynomial ring.
Put $\dd_i=\frac {\dd}{\partial t_i}$, $i\ge 0$.
Define the following two derivations of $R$:
\begin{align*}
v_1 & =\dd_1+t_0(\dd_2+t_1(\dd_3+t_2(\dd_4+t_3(\dd_5+t_4(\dd_6+\cdots )))));\\
v_2 & =\qquad\quad\;\,
\dd_2+t_1(\dd_3+t_2(\dd_4+t_3(\dd_5+t_4(\dd_6+\cdots )))).
\end{align*}
These two derivations generate
a restricted Lie algebra $\LL=\Lie_p(v_1,v_2)\subset\Der R$ and an associative algebra $\AA=\Alg(v_1,v_2)\subset \End R$.
By Bergman's theorem, the Gelfand-Kirillov dimension of an associative algebra cannot belong to the interval $(1,2)$~\cite{KraLen}.
Such a gap for Lie algebras does not exist, the Gelfand-Kirillov
dimension of a finitely generated Lie algebra can be arbitrary number $\{0\}\cup [1,+\infty)$~\cite{Pe97}.
The Fibonacci Lie algebra has slow polynomial growth
with Gelfand-Kirillov dimension $\GKdim \LL=\log_{(\sqrt 5+1)/2} 2\approx 1.44$~\cite{Pe06}.
Further properties of the Fibonacci restricted Lie algebra and its generalizations are studied in~\cite{PeSh09,PeSh13fib}.

Probably, the most interesting property of $\LL$ is that it has a nil $p$-mapping~\cite{Pe06},
which is an analog of the periodicity of the Grigorchuk and Gupta-Sidki groups.
We do not know whether the associative hull $\AA$ is a nil-algebra.
We have a weaker statement.
The algebras $\LL$, $\AA$, and the augmentation ideal
of the restricted enveloping algebra
$\uu=\omega u(\LL)$ are direct sums of two locally nilpotent subalgebras~\cite{PeSh09}:
\begin{equation}\label{decompLL}
\LL=\LL_+\oplus \LL_-,\quad \AA=\AA_+\oplus \AA_-,\quad \uu=\uu_+\oplus \uu_-.
\end{equation}
There are examples of infinite dimensional associative algebras which
are direct sums of two locally nilpotent subalgebras~\cite{Kel93,DreHam04}.
Infinite dimensional restricted Lie algebras can have
different decompositions into a direct sum of two locally nilpotent subalgebras~\cite{PeShZe10}.
\medskip

In case of arbitrary prime characteristic,
Shestakov and Zelmanov suggested an example of a finitely generated restricted Lie algebra
with a nil $p$-mapping~\cite{ShZe08}.
That example yields the same decompositions~\eqref{decompLL} for some primes~\cite{Kry11,PeSh13ass}.
An example of a $p$-generated  nil restricted Lie algebra $L$,
characteristic $p$ being arbitrary, was studied in~\cite{PeShZe10}.
The virtue of that example is that for all primes $p$
we have the same decompositions~\eqref{decompLL} into direct sums of two locally nilpotent subalgebras.
But computations for that example are rather complicated.

Observe that only the original example has a clear monomial basis~\cite{Pe06,PeSh09}.
In other examples, elements of a Lie algebra are  linear combinations of monomials,
to work with such linear combinations is sometimes an essential technical difficulty, see e.g.~\cite{ShZe08,PeShZe10}.
A family of nil restricted Lie algebras of slow polynomial growth having clear monomial bases
is constructed in~\cite{Pe17},
these algebras are close relatives of a two-generated Lie superalgebra of~\cite{Pe16}.
As a particular case, we obtain a continuum subfamily
of non-isomorphic nil restricted Lie algebras with
the Gelfand-Kirillov dimension equal to one,  but their growth is not linear~\cite{Pe16}.

\subsection{Narrow groups and Lie algebras}\label{SSnarrow}
Let $G$ be a group and $G=G_1\supseteq G_2\supseteq \cdots$ its lower central series.
One constructs a related $\N$-graded Lie algebra
$L_K(G)=\oplus_{i\ge 1} L_i$, where $L_i= G_i/G_{i+1}\otimes_{\Z} K$, $i\ge 1$.
A product is given by $[a_iG_{i+1},b_jG_{j+1}]=(a_i,b_j)G_{i+j+1}$,
where $(a_i,b_j)=a_i^{-1}b_j^{-1}a_ib_j$ is the group commutator.

A residually $p$-group $G$ is said of {\it finite width} if
all factors $G_i/G_{i+1}$ are finite groups with uniformly bounded orders.
The Grigorchuk group $G$ is of finite width, namely,
$\dim_{\F_2} G_i/G_{i+1}\in\{1,2\}$ for $i\ge 2$~\cite{Rozh96,BaGr00}.
In particular, the respective Lie algebra $L=L_K(G)=\oplus_{i\ge 1} L_i$ has a linear growth.
Bartholdi presented $L_{K}(G)$ as a self-similar restricted Lie algebra
and proved that the restricted Lie algebra $L_{\F_2}(G)$ is nil while $L_{\F_4}(G)$ is not nil~\cite{Bartholdi15}.
Also, $L_K(G)$ is {\it nil graded}, namely,
for any homogeneous element $x\in L_i$, $i\ge 1$, the mapping $\ad x$ is nilpotent,
because $G$ is periodic.

Shalev and Zelmanov suggested that narrowness
conditions for Lie algebras deserve systematic study because the most
important algebras are narrow in some sense~\cite{ShaZel97, ShaZel99}.
Infinite dimensional $\N$-graded Lie algebras $L=\mathop{\oplus}\limits_{n=1}^\infty L_n$
with one-dimensional components in characteristic zero were classified by Fialowski~\cite{Fial83}.
A Lie algebra $L$ is called of {\it maximal class} (or {\it filiform}),
if the associated graded algebra with respect to the lower central series
$\gr L=\mathop{\oplus}\limits_{n=1}^\infty \gr L_n$, where $\gr L_n=L^n/L^{n+1}$, $n\ge 1$,
satisfies
\begin{equation}\label{filiform}
    \dim \gr L_1=2,\quad \dim \gr L_n\le 1,\ n\ge 2,\quad \gr L_{n+1}=[\gr L_1, \gr L_n],\ n \ge 1,
\end{equation}
in particular, $\gr L$ is generated by $\gr L_1$.
It follows from the paper of Vergne~\cite{Vergne70} that in zero characteristic
there exists a unique infinite dimensional $\N$-graded Lie algebra whose grading satisfies~\eqref{filiform}.
In zero characteristic their exist uncountably many pairwise non-isomorphic
$\N$-graded filiform Lie algebras  of width one in dimensions 7--11~\cite{Mil04}.

An infinite dimensional filiform Lie algebra $L$ has the
smallest nontrivial growth function: $\gamma_L(n)=n+1$, $n\ge 1$.
In case of positive characteristic, there are uncountably many such algebras~\cite{CarMatNew97}.
Nevertheless, in case $p>2$, they were classified in~\cite{CarNew00}.
There are generalizations  of filiform Lie algebras.
A graded Lie algebra $L=\mathop{\oplus}\limits_{n=1}^\infty L_n$ is said {\it thin} if
$\dim L_1=2$ and it  satisfies the {\it covering property}:
$L_{n+1}=[L_1, z]$ for any $0\ne z\in L_n$ and all $n\ge 1$~\cite{CarMatNewSco96}.
In particular, $\dim L_n\le 2$ for all $n\ge 1$.
Components of dimension~2 are called {\it diamonds}.
This notion appeared in a theory of similarly defined {\it thin groups}~\cite{CarMatNewSco96}.

Naturally $\N$-graded Lie algebras over $\R$ and $\C$
satisfying the condition $\dim L_n+\dim L_{n+1}\le 3$, $n\ge 1$, are classified recently by Millionschikov~\cite{Mil}.
More generally, an $\N$-graded Lie algebra
$L=\mathop{\oplus}\limits_{n=1}^\infty L_n$ is said of finite {\it width} $d$ in the case that $\dim L_n\le d$, $n\ge 1$,
the integer $d$ being minimal.

Pro-$p$-groups and $\N$-graded Lie algebras cannot be simple.
Instead one has an important notion of being {\it just infinite}, namely,
not having non-trivial normal subgroups (ideals) of infinite index (codimension).
A group (algebra) is said {\it hereditary just infinite}
if and only if any normal subgroup (ideal) of finite index (codimension) is just infinite.
The Gupta-Sidki groups were the first in the class of periodic groups to be shown to be just infinite~\cite{GuptaSidki83A}.
The Grigorchuk group is also just infinite but not hereditary just infinite~\cite{Grigorchuk00horizons}.

\subsection{Lie algebras over a field of characteristic zero}
Since the Grigorchuk group is periodic and of finite width,
its appropriate analogue should be a "small" $\ad$-nil Lie algebra.
For example, a Lie algebra should be  of finite width, implying a linear growth.
In the next result, a weaker assumption that components of a grading are bounded
implies that the growth of finitely generated subalgebras is polynomial.
Informally speaking, there are no "natural analogues" of the Grigorchuk and Gupta-Sidki groups
in the world of Lie algebras of characteristic zero
in terms of the following result {\it sensu stricto}.

\begin{Theorem}[{Martinez and Zelmanov~\cite{MaZe99}}]
\label{TMarZel}
Let $L=\oplus_{\a\in\Gamma}L_\alpha$ be a Lie algebra over a field $K$ of
characteristic zero graded by an abelian group $\Gamma$. Suppose that
\begin{enumerate}
\item
there exists $d>0$ such that $\dim_K L_\alpha \le d $ for all $\alpha\in\Gamma$,
\item
every homogeneous element $a\in L_\a$, $\a\in\Gamma$, is ad-nilpotent.
\end{enumerate}
Then the Lie algebra $L$ is locally nilpotent.
\end{Theorem}

\subsection{Fractal nil graded Lie superalgebras $\RR$ and $\QQ$}
In the world of {\it Lie  superalgebras} of an {\it arbitrary characteristic},
the second author constructed analogues of the Grigorchuk and Gupta-Sidki groups~\cite{Pe16}.
Namely, two Lie superalgebras $\RR$, $\QQ$ were constructed,
which are also analogues of the Fibonacci restricted Lie algebra and other restricted Lie algebras mentioned above.
The constructions of both Lie superalgebras $\RR$, $\QQ$ are similar,
computations for $\RR$ are simpler, but $\QQ$ enjoys some more specific interesting properties.
The virtue of both examples is that they have clear monomial bases.
They have slow polynomial growth, namely,
$\GKdim \RR=\log_34\approx 1.26$ and $\GKdim \QQ=\log_38\approx 1.89$.
Thus, both Lie superalgebras are of infinite width.
In both examples, $\ad a$ is nilpotent,
$a$ being an even or odd element with respect to $\Z_2$-gradings as Lie superalgebras.
This property is an analogue of the periodicity of the Grigorchuk and Gupta-Sidki groups.
The Lie superalgebra $\RR$ is $\Z^2$-graded, while $\QQ$ has a natural fine $\Z^3$-gradation
with at most one-dimensional components
(See on importance of fine gradins for Lie and associative algebras~\cite{BaSeZa01,Eld10}.
There are examples of simple $\Z^2$-graded Lie algebras in characteristic zero with
one-dimensional components~\cite{IohMat13}).
In particular, $\QQ$ is a nil finely graded Lie superalgebra, which shows
that an extension of Theorem~\ref{TMarZel} (Martinez and Zelmanov~\cite{MaZe99})
for the Lie superalgebras of characteristic zero is not valid.
Also, $\QQ$ has a $\Z^2$-gradation which yields a continuum of different decompositions
into sums of two locally nilpotent subalgebras $\QQ=\QQ_+\oplus\QQ_-$.
Both Lie superalgebras are {\it self-similar}.
They also contain infinitely many copies of themselves, we call them {\it fractal} due to the last property.
\medskip

In Section~\ref{Sdef} we supply basic definitions.
In Section~\ref{Smain} we give a construction of our Lie superalgebra $\RR$
(this notation was also used for the 2-generated Lie superalgebra~$\RR$ of~\cite{Pe16}, which shall be
mentioned in the present paper one more time) and formulate its basic properties.

\section{Basic definitions: (restricted) Lie  superalgebras, growth}\label{Sdef}

Superalgebras appear naturally in physics and mathematics~\cite{Kac77,Scheunert,BMPZ}.
By $K$ denote the ground field, $\langle S\rangle_K$ a linear span of a subset $S$ in a $K$-vector space.
Put $\Z_2=\{\bar 0,\bar 1\}$, the group of order 2.
A {\em superalgebra} $A$ is a $\Z_2$-graded algebra $A=A_{\bar 0}\oplus A_{\bar 1}$.
The elements $a\in A_\alpha$ are called {\em homogeneous of degree} $\deg a=\alpha$, $\alpha\in\Z_2$.
The elements of $A_{\bar 0}$ are {\em even}, those of $A_{\bar 1}$ {\em odd}.
In what follows, if $\deg a$ enters an expression,
then it is assumed that $a$ is homogeneous of degree $\deg a\in\Z_2$,
and the expression extends to the other elements by linearity.
Let $A,B$ be superalgebras, a {\em tensor product} $A\otimes B$ is the superalgebra
whose space is the tensor product of the spaces $A$ and $B$ with the induced $\Z_2$-grading and the product:
$$
(a_1\otimes b_1) (a_2\otimes b_2)=(-1)^{\deg b_1\cdot \deg a_2}a_1a_2\otimes b_1b_2,\quad a_i\in A,\ b_i\in B.
$$
An {\em associative superalgebra} $A$ is a
$\Z_2$-graded associative algebra $A=A_{\bar 0}\oplus A_{\bar 1}$.
A {\em Lie superalgebra} is a $\Z_2$-graded algebra $L=L_{\bar 0}\oplus L_{\bar 1}$ with an
operation $[\ ,\ ]$ satisfying the axioms:
\begin{itemize}
\item
$[x,y]=-(-1)^{\deg x\cdot \deg y }[y,x]$,\quad (super-anticommutativity);
\item
$[x,[y,z]]=[[x,y],z]+(-1)^{\deg x\cdot\deg y}[y,[x,z]]$,\quad (Jacobi identity).
\end{itemize}
All commutators in the present paper are supercommutators.
Long commutators are {\it right-normed}: $[x,y,z]=[x,[y,z]]$.
We use a standard notation $\ad x(y)=[x,y]$, where $x,y\in L$.

Assume that $A=A_{\bar 0}\oplus A_{\bar 1}$ is an associative superalgebra.
One obtains a Lie superalgebra $A^{(-)}$
by supplying the same vector space $A$  with a {\em supercommutator}:
$$
[x,y]=xy-(-1)^{\deg x\cdot \deg y}yx,\quad x,y\in A.
$$
If 
$A^{(-)}$ is abelian,
then 
$A$ is called {\it supercommutative}.
Let $L$ be a Lie superalgebra,
one defines a {\it universal enveloping algebra}
$U(L)=T(L)/(x\otimes y- (-1)^{\deg x\cdot\deg y}y\otimes x-[x,y]\mid x,y\in L)$, where
$T(L)$ is the tensor algebra of the vector space $L$.
Now, the product in $L$ coincides with the supercommutator in $U(L)^{(-)}$.
A basis of  $U(L)$ is given by PBW-theorem~\cite{BMPZ,Scheunert}.

In case $\ch K=2,3$ the axioms of the Lie superalgebra have to be augmented
(\cite[section 1.10]{BMPZ}, \cite{Bou-Leites-09}, \cite{Pe16}).
\begin{itemize}
\item ($\ch K= 3$)
$[z,[z,z]]= 0$,  $z\in L_{\bar 1}$.
\end{itemize}
Substituting $x=y\in L_{\bar 1}$ in the Jacobi identity,  we get $2 (\ad x)^2 z=[[x,x], z]$.
In case $\ch K\ne 2$ we get an identity
\begin{equation*}
(\ad x)^2 z= \frac 12 [[x,x],z],\qquad  x\in L_{\bar 1}, \ z\in L.
\end{equation*}
In the present paper we study Lie superalgebras of the form $A^{(-)}$,
they have squares: $[x,x]=2 x^2$, $x\in A_{\bar 1}^{(-)}$.
One obtains an identity that is also valid for $\ch K=2$:
\begin{equation}\label{squares}
(\ad x)^2 z= [x^2,z] 
,\qquad x\in A^{(-)}_{\bar 1},\ z\in A^{(-)}.
\end{equation}
In case $\ch K=2$, we add more axioms for the Lie superalgebras:
\begin{itemize}
\item there exists a {\em quadratic mapping (a formal square)}:
$(\ )^{[2]}:L_{\bar 1}\to L_{\bar 0}$, $x\mapsto x^{[2]}$, $x\in L_{\bar 1}$, satisfying:
\begin{align}
 (\lambda x)^{[2]}&=\lambda^2 x^{[2]},\qquad x\in L_{\bar 1},\ \lambda\in K;\nonumber\\
 (x+y)^{[2]}&=x^{[2]}+[x,y]+y^{[2]};\quad x,y\in L_{\bar 1}; \label{square2}\\
 (\ad x)^2 z &= [x^{[2]},z],\qquad  x\in L_{\bar 1},\  z\in L,\
 \text{(a formal substitute of~\eqref{squares})};\nonumber
\end{align}
\item $[x,x]=0$, $x\in L_{\bar 0}$.
By putting $y=x$ in the second relation above, we get $[y,y]=0$, $y\in L_{\bar 1}$.
\end{itemize}
In other words, a Lie superalgebra in case $\ch K=2$ is just a $\Z_2$-graded Lie algebra
supplied with a quadratic mapping
$L_{\bar 1}\to L_{\bar 0}$, which is similar to the  $p$-mapping (see below).
In case $p=2$, to get the universal enveloping algebra,
we additionally factor out $\{y\otimes y-y^{[2]}\mid y\in L_{\bar 1}\}$.

We shall consider also {\it Lie super-rings}. These are $\Z_2$-graded $\Z$-algebras whose
product satisfies the super-anticommutativity and the Jacobi identity.
Since Lie algebras over a field of $\ch K=2$ should be considered as such rings,
we also assume existence of the formal quadratic mapping satisfying the axioms above.

In all cases, the quadratic mapping will be denoted by $x^{2}$, $x\in L_{\bar 1}$,
it also coincides with the ordinary square in the universal (or restricted) enveloping algebra $U(L)$.

Let $V=V_{\bar 0}\oplus V_{\bar 1}$ be a vector space, we say that it is $\Z_2$-graded.
The associative algebra
of all endomorphisms of the space
$\End V$ is an associative superalgebra:
$\End V=\End_{\bar 0} V\oplus \End_{\bar 1} V$,
where $\End_\alpha V=\{\phi\in\End V\mid \phi(V_\beta)\subset V_{\alpha+\beta}, \beta\in\Z_2\}$, $\alpha\in\Z_2$.
Thus, $\End^{(-)}V$ is a Lie superalgebra, called the {\em general linear superalgebra} ${gl}(V)$.

Let $A=A_{\bar 0}\oplus A_{\bar 1}$ be a $\Z_2$-graded algebra of arbitrary signature.
A linear mapping $\phi\in\End_{\beta} A$, $\beta\in\Z_2$,
is a {\em superderivative} of degree $\beta$ if it satisfies
\begin{equation*}
\phi(a\cdot b)=\phi(a)\cdot b+(-1)^{\beta\cdot\deg(a)}a\cdot\phi(b),\quad a,b\in A.
\end{equation*}
Denote by $\Der_{\alpha}A\subset \End_{\alpha} A$ the space of all superderivatives of degree $\alpha\in\Z_2$.
One checks that $\Der A=\Der_{\bar 0} A\oplus \Der_{\bar 1} A$
is a subalgebra of the Lie superalgebra $\End^{(-)}A$.
In this paper by a derivation we always mean a superderivation.

\medskip

Let $L$ be a Lie algebra over a field $K$ of characteristic $p>0$.
Then $L$ is called a
\textit{restricted Lie algebra} (or \textit{Lie $p$-algebra}),
if it is additionally supplied with a unary operation
 $x\mapsto x^{[p]}$, $x\in L$, that satisfies the following axioms~\cite{JacLie,Strade1,StrFar}:
\begin{itemize}
\item $(\lambda x)^{[p]}=\lambda^px^{[p]}$, for $\lambda\in K$, $x\in L$;
\item $\ad(x^{[p]})=(\ad x)^p$, $x\in L$;
\item $(x+y)^{[p]}=x^{[p]}+y^{[p]}+\sum_{i=1}^{p-1}s_i(x,y)$, $x,y\in L$,
where $is_i(x,y)$~is the coefficient of $t^{i-1}$ in the polynomial
$\operatorname{ad}(tx+y)^{p-1}(x)\in L[t]$.
\end{itemize}
This notion is motivated by the following observation.
Let $A$ be an associative algebra over a field ~$K$, $\operatorname{char}K=p>0$.
Then the mapping  $x\mapsto x^p$, $x\in A^{(-)}$,
satisfies these conditions considered in the Lie algebra $A^{(-)}$.

A {\em restricted Lie superalgebra} $L=L_{\bar 0}\oplus L_{\bar 1}$ is a Lie superalgebra
such that the even component $L_{\bar 0}$ is a restricted Lie algebra and $L_{\bar 0}$-module
$L_{\bar 1}$ is restricted, i.e. $\ad(x^{[p]}) y=(\ad x)^p y$, for all $x\in L_{\bar 0}$, $y\in L_{\bar 1}$
(see. e.g.~\cite{Mikh88,BMPZ,Pe92}).
Remark that in case $\ch K=2$, the restricted Lie superalgebras and
$\Z_2$-graded restricted Lie algebras are the same objects.
(Let $L=L_{\bar 0}\oplus L_{\bar 1}$ be a restricted Lie superalgebra,
it has the $p$-mapping on the even part: $L_{\bar 0}\to L_{\bar 0}$ and the formal square
on the odd part: $L_{\bar 1}\to L_{\bar 0}$.
We obtain a $p$-mapping
on the whole of algebra by setting $(x+y)^{[2]}=x^{[2]}+y^{[2]}+[x,y]$, $x\in L_{\bar 0}$,  $y\in L_{\bar 1}$).

Let $L$ be a restricted Lie (super)algebra, and
$J$ the ideal of the universal enveloping algebra~$U(L)$
generated by all elements $x^{[p]}-x^p$, $x\in L_{\bar 0}$.
Then $u(L)=U(L)/J$ is the \textit{restricted enveloping algebra} of $L$.
In this algebra, the formal operation $x^{[p]}$ coincides with the ordinary power $x^p$ for all $x\in L_{\bar 0}$.
One has an analogue of PBW-theorem
describing a basis of $u(L)$ \cite[p.~213]{JacLie}, \cite{BMPZ}.

Let $L$ be a Lie (super)algebra. One defines the {\it lower central series} as
$L^1=L$ and $L^{n}=[L,L^{n-1}]$, $n\ge 2$. In case $\ch K=2$ the space above
is augmented by $\langle x^2\mid x\in (L^{[n/2]})_{\bar 1}\rangle_K$.
In case of a restricted Lie (super)algebra,
we also add $\langle x^p\mid x\in (L^{[n/p]})_{\bar 0}\rangle_K$.
\medskip

We recall the notion of {\em growth}. Let $A$  be an associative (or Lie) (super)algebra  generated by a finite set $X$.
Denote  by $A^{(X,n)}$ the subspace of $A$ spanned by all  monomials  in $X$
of length not  exceeding  $n$, $n\ge 0$.
In case of a Lie superalgebra of $\ch K=2$ we also consider formal squares of odd monomials of length at most $n/2$.
If $A$ is a restricted Lie algebra, put
$A^{(X,n)}=\langle\, [x_{1},\dots,x_{s}]^{p^k}\mid x_{i}\in X,\, sp^k\le n\rangle_K$~\cite{Pape01}.
Similarly, one defines the growth for restricted Lie superalgebras.
In either situation, one defines an  {\em (ordinary) growth function}:
$$
\gamma_A(n)=\gamma_A(X,n)=\dim_KA^{(X,n)},\quad n\ge 0.
$$
Let $f,g:\N\to\R^+$ be eventually increasing and positive valued functions.
Write $f(n)\preccurlyeq g(n)$ if and only if there exist positive constants $N,C$ such that $f(n)\le g(Cn)$
for all $n\ge N$.
Introduce equivalence $f(n)\sim g(n)$ if and only if  $f(n)\preccurlyeq g(n)$ and $g(n)\preccurlyeq f(n)$.
Different generating sets of an algebra yield equivalent growth functions~\cite{KraLen}.

It is well known that the
exponential growth is the highest possible growth for finitely generated Lie and
associative algebras. A growth function $\gamma_A(n)$ is
compared with polynomial functions $n^\alpha$, $\alpha\in\R^+$, by
computing the {\em upper and lower Gelfand-Kirillov
dimensions}~\cite{KraLen}:
\begin{align*}
\GKdim A&=\limsup_{n\to\infty} \frac{\ln\gamma_A(n)}{\ln n}=\inf\{\a>0\mid \gamma_A(n)\preccurlyeq n^\a\} ;\\
\LGKdim A&=\liminf_{n\to\infty}\,  \frac{\ln\gamma_A(n)}{\ln n}=\sup\{\a>0\mid \gamma_A(n)\succcurlyeq n^\a\}.
\end{align*}

Assume that generators $X=\{x_1,\dots,x_k\}$ are assigned positive weights $\wt(x_i)=\lambda_i$, $i=1,\dots,k$.
Define a {\it weight growth function}:
$$
\tilde \gamma_A(n)=\dim_K\langle x_{i_1}\cdots x_{i_m}\mid \wt(x_{i_1})+\cdots+\wt(x_{i_m})\le n,\
          x_{i_j}\in X\rangle_K,\quad n\ge 0.
$$
Set $C_1=\min\{\lambda_i\mid i=1,\dots,k \}$, $C_2=\max\{\lambda_i\mid i=1,\dots,k \}$,
then $\tilde\gamma_A(C_1 n) \le \gamma_A(n)\le \tilde\gamma_A(C_2 n)$ for $n\ge 1$.
Thus, we obtain an equivalent growth function $\tilde \gamma_A(n)\sim\gamma_A(n)$.
Therefore, we can use the weight growth function $\tilde\gamma_A(n)$ in order to
compute the Gelfand-Kirillov dimensions.
By $f(n)\approx g(n)$, $n\to\infty$, denote that $\mathop{\lim}\limits_{n\to\infty} f(n)/g(n)=1$.

Suppose that $L$ is a Lie (super)algebra and $X\subset L$.
By $\Lie(X)$ denote the subalgebra of $L$ generated by $X$
(including application of the quadratic mapping in case $\ch K=2$).
Let $L$ be a restricted Lie (super)algebra,
by $\Lie_p(X)$ denote the restricted subalgebra of $L$ generated by $X$.
Assume that $X$ is a subset of an associative algebra $A$.
Write $\Alg(X)\subset A$ to denote the associative subalgebra (without unit) generated by~$X$.
A grading of an algebra is called {\em fine} if
it cannot be splitted by taking a bigger grading group (see definitions in~\cite{BaSeZa01,Eld10}).

Assume that $I$ is a well-ordered set of arbitrary cardinality. Put $\Z_2=\{0,1\}$.
Let $\Z_2^I=\{\a: I\to\Z_2\}$ be the set of functions with finitely many nonzero values.
Suppose that $\a \in \Z_2^I$ has nonzero values only at $\{i_1,\dots,i_t\}\subset I$, where $i_1<\cdots< i_t$, put
$\mathbf x^\a=x_{i_1}x_{i_2}\cdots x_{i_t}$ and $|\a|=t$.
Now $\{ \mathbf x^\a\mid \a\in \Z_2^I\}$ is a basis of the Grassmann algebra $\Lambda_I=\Lambda(x_i\mid i\in I)$,
which is an associative superalgebra $\Lambda_I=\Lambda_{\bar 0}\oplus \Lambda_{\bar 1}$,
all $x_i$, $i\in I$, being odd.
Let $\partial_i$, $i\in I$, denote the superderivatives of
$\Lambda$ determined by  $\partial_i(x_j)=\delta_{ij}$, $i,j\in I$.
We identify $x_i$, $i\in I$, with operators
of left multiplication on $\Lambda_I$, thus we consider that
 $x_i\in \End(\Lambda_I)$, $i\in I$.
Consider the space of all formal sums
\begin{equation}\label{WW}
\WW(\Lambda_I)=\bigg\{\sum_{\a\in \Z_2^I} {\mathbf x}^{\a}\sum_{j=1}^{m(\a)}
                           \lambda_{\a,i_j}\,\partial_{i_j}
                    \ \bigg|\ \lambda_{\a,i_j}\in K,\ i_j\in I
                    \bigg\}.
\end{equation}
It is essential that the sum at each ${\mathbf x}^{\a}$, $\a\in \Z_2^I$,  is finite.
This construction is similar to so called Lie algebra of {\em special derivations},
see~\cite{Rad86}, \cite{Razmyslov}, \cite{PeRaSh}.
It is similarly verified that
the product on $\WW(\Lambda_I)$ is well defined and
the Lie superalgebra $\WW(\Lambda_I)$ acts on $\Lambda_I$ by superderivations.

\section{Main results: Lie superalgebra $\RR$ and its properties}\label{Smain}

Let $\Lambda=\Lambda(x_i\vert i\geq 0)$ be the Grassmann algebra.
The Grassmann letters and respective superderivatives
$\lbrace x_i,\dd_i\mid i\geq 0\rbrace$ are odd elements of the superalgebra $\End\Lambda$.
One has relations for all $i,j\ge 0$:
\begin{align*}
x_i x_j&=-x_j x_i, \quad  \dd_i\dd_j=-\dd_j\dd_i, \quad
     \dd_ix_j=-x_j\dd_i, \qquad i\neq j; \\
x_i^2&=\dd_i^2=0; \qquad
\dd_i x_i+x_i\dd_i=1.
\end{align*}
Our superalgebras will be constructed as subalgebras in $\WW(\Lambda)\subset \Der \Lambda$
or $\End(\Lambda)$.
Our goal is to study the following finitely generated (restricted) Lie superalgebra
and its associative hull.
\medskip

{\bf Example. }{\it
Let $\Lambda=\Lambda(x_i \mid i\ge 0)$. Consider the following elements in $\WW(\Lambda)$:
\begin{equation}\label{pivot}
v_i=\dd_i+x_{i}x_{i+1}(\dd_{i+2}+x_{i+2}x_{i+3}(\dd_{i+4}+x_{i+4}x_{i+5}(\dd_{i+6}+\ldots))),\qquad i\geq 0.
\end{equation}
We define a Lie superalgebra $\RR=\Lie(v_0,v_1)\subset\WW(\Lambda)$ and its associative hull
$\AA=\Alg(v_0,v_1)\subset \End \Lambda$.
In case $\ch K=2$, we assume that the quadratic mapping on odd elements is the square of
the respective operator in $\End\Lambda$.
} 
\medskip

We call $\{v_i\mid i\ge 0\}$ {\em pivot elements}, observe that they are odd in the Lie superalgebra $\WW(\Lambda)$.
We establish the following properties of $\RR=\Lie(v_0,v_1)$
and its associative hull $\AA=\Alg(v_0,v_1)$.

\begin{enumerate}
\renewcommand{\theenumi}{\roman{enumi}}
\item Section~\ref{Srelations} yields basic relations of $\RR$. 
\item $\RR$ has a monomial basis consisting of standard monomials of two types ($\ch K\ne 2$, Theorem~\ref{Tbase}).
      In case $\ch K=2$, a basis of $\RR$ is given by monomials of the first type
      and squares of the pivot elements (Corollary~\ref{Cchar2}), and $\RR$
      coincides with the restricted Lie algebra $\Lie_p(v_0,v_1)$.
\item We introduce two weight functions $\wt(\ )$, $\swt(\ )$  that are additive on products of monomials.
      Using these functions, we prove that $\RR$ and $\AA$ are $\Z^2$-graded
      by multidegree in the generators (Lemma~\ref{Lz2graduacao}).
      This allows us to introduce two coordinate systems on plane:
      multidegree coordinates $(X_1,X_2)$ and weight coordinates $(Z_1,Z_2)$.
      We introduce a weight $\N$-gradation and a degree (natural) $\N$-gradation,
      which components are also the factors of the lower central series
      (Section~\ref{Sweight}).
\item We find bounds on weights and superweights of the basis monomials of $\RR$ and $\AA$
      (Section~\ref{Sweightbounds} and Section~\ref{SweightboundsA})
      and prove that their monomials are in regions of plane
      bounded by pairs of logarithmic curves (Theorem~\ref{TcurvesL}, Figure~\ref{Fig1}, and Theorem~\ref{TcurvesA}).
\item The components of the $\Z^2$-grading of $\RR$ are at most one-dimensional (Theorem~\ref{Tfine}),
      thus, the $\Z^2$-grading of $\RR$ is fine.
      Almost all components of the weight $\N$-gradation of $\RR$ are two-dimensional (Corollary~\ref{Cnaturgrad}).
\item $\GKdim\RR=\LGKdim\RR=1$, moreover, $\RR$ has a linear growth.
      In case $\ch K\ne 2$, we establish an asymptotic for the ordinary growth function:
      $\gamma_\RR(m)\approx 3m$, as $m\to\infty$ (Theorem~\ref{TRgrowth}).
\item Moreover, let $\RR=\mathop{\oplus}\limits_{n=1}^\infty\RR_n$ be the $\N$-grading by degree in the generators,
      where also $\RR_n\cong \RR^n/\RR^{n+1}$, $n\ge 1$, are the lower central series factors.
      In case $\ch K\ne 2$, we prove that  $\RR$ is of finite width 4, namely,
      the coefficients $(\dim\RR_n| n\ge 1)$, are $\{2,3,4\}$ (Theorem~\ref{Twidth}).
      (Recall that both self-similar Lie superalgebras in~\cite{Pe16} were of infinite width).
      This is analogous to the fact that the Grigorchuk group is of finite width~\cite{Rozh96,BaGr00}.
\item Let $\ch K=2$, $\RR$ the Lie algebra generated by $\{v_0,v_1\}$ (i.e. without $p$-mapping), and
      $\RR=\mathop{\oplus}\limits_{n=1}^\infty\RR_n$ the $\N$-grading by degree in the generators.
      Then $\RR$ is of width~2. Moreover,
      the sequence $(\dim\RR_n | n\ge 1)$, starting with $n\ge 5$,
      consists of alternating parts of two types:
      a block $1,1$ followed either by 2 (a diamond) or by $2,2,2$ (a triplet of diamonds) (Corollary~\ref{Cdiamonds}).
      But, $\RR$ is not a thin Lie algebra and
      the sequence above is  eventually non-periodic.
\item Unlike many examples of (self-similar) (restricted) Lie (super)algebras studied before,
      this is the first example that we are able to find a clear monomial basis of the associative hull $\AA$ (Theorem~\ref{TbasisA}).
\item $\GKdim\AA=\LGKdim\AA=2$, moreover, we establish quadratic bounds on the growth functions
      (Theorem~\ref{TgrowthA}).
\item We study generating functions of $\RR$  (Section~\ref{Sfunctions}).
      The results and proofs on basis monomials of $\RR$ are illustrated by Figure~\ref{Fig1} and Figure~\ref{Fig2}.
\item As an analogue of the periodicity,
      we establish that homogeneous elements of
      the grading $\mathbf{R}=\mathbf{R}_{\bar 0}\oplus\mathbf{R}_{\bar 1}$ are $\ad$-nilpotent
      (Theorem~\ref{Tadnilpotente}).
      We have a triangular decomposition into a direct sum of three locally nilpotent subalgebras
      $\RR=\RR_+\oplus \RR_0\oplus \RR_-$ (Theorem~\ref{Tlocal}).
      In case $\ch K=0$ or $\ch K=2$, the associative hull $\AA$ is not nil;
      in case $\ch K=2$, the restricted Lie algebra $\RR=\Lie_p(v_0,v_1)$ is not nil (Lemma~\ref{Lnaonil}).
\item $\RR$ is just infinite~(Theorem~\ref{Tjustinf}) but not hereditary just infinite
      (Lemma~\ref{Lnonjustinf}).
      (Whether two Lie superalgebras of~\cite{Pe16} are just infinite is not known).
\item $\RR$ shows that an extension of~Theorem~\ref{TMarZel} (Martinez and Zelmanov~\cite{MaZe99})
      to the Lie superalgebras of characteristic zero is not valid.
      Such a counterexample of a nil finely $\Z^3$-graded Lie superalgebra
      of slow polynomial growth $\QQ$ was suggested before~\cite{Pe16}.
      The present example is more handy, because the Lie superalgebra $\RR$ is of linear growth, moreover, of finite width 4, and just infinite.
\end{enumerate}

The research is continued in~\cite{PeSh18-Jor},
where we study a different
just infinite, $\ad$-nil, finely $\Z^3$-graded, fractal, 3-generated Lie superalgebra and
suggest new constructions yielding interesting related nil graded fractal Poisson and Jordan superalgebras,
that are natural analogues of the Grigorchuk and Gupta-Sidki groups in that classes of algebras.

One can also apply that constructions to the present Lie superalgebra $\RR$
and obtain a Jordan superalgebra $\mathbf K$, that will be
3-generated, just infinite, nil $\Z^3$-graded, with at most one-dimensional $\Z^3$-components, of linear growth,
moreover, of finite width~4, namely, its $\N$-gradation by degree in the generators has a non-periodic pattern
of components of dimensions $\{0,2,3,4\}$, the field being arbitrary with $\ch K\ne 2$.
\section{Relations of Lie superalgebra $\RR$}\label{Srelations}

Recall that we consider the Grassmann algebra in infinitely many variables $\Lambda=\Lambda(x_i \mid i\ge 0)$
and its odd derivatives~\eqref{pivot}, called the pivot elements. We also present them recursively:
\begin{eqnarray}\label{pivot2}
v_i=\dd_i+x_{i}x_{i+1}v_{i+2},\qquad i\geq 0.
\end{eqnarray}
Below we use a notation $x_i\cdots x_j=x_{i}x_{i+1}\cdots x_{j-1}x_j$ in case $i\le j$,
otherwise the product equals 1.
Let $0\le n<m$ and $m-n$ is even. We extend the presentation above as:
\begin{equation}
\label{pivot3}
v_n=\dd_n+x_n x_{n+1}(\dd_{n+2}+x_{n+2}x_{n+3}(\dd_{n+4}+\ldots+x_{m-4}x_{m-3}(\dd_{m-2}+x_{m-2}x_{m-1}v_m)\cdots)).
\end{equation}
Observe that the action of
the pivot elements on the Grassmann letters can only produce smaller letters:
\begin{equation}\label{action}
v_n(x_k)=
\begin{cases}
0, & k<n;\\
1, & k=n;\\
0, & k=n+2l-1, \qquad l\ge 1;\\
x_n x_{n+1}\cdots x_{k-2}x_{k-1},\quad & k=n+2l, \qquad\qquad l\ge 1.
\end{cases}
\end{equation}
Define a {\it shift mapping}
$\tau:\Lambda\rightarrow\Lambda$, $\tau:\WW(\Lambda)\rightarrow\WW(\Lambda)$ by
$\tau(x_i)=x_{i+1}$ and $\tau(\dd_i)=\dd_{i+1}$ for $i\geq 0$.
Clearly, $\tau$ is an endomorphism such that $\tau(v_i)=v_{i+1}$, $i\ge 0$.

First, we obtain relations between neighbor pivot elements.
\begin{Lemma}\label{Lbasicas}
We have the following relations, $i\geq 0$:
\begin{enumerate}
\item $v_i^{2}=x_{i+1}v_{i+2}$;
\item $[v_i,v_i]=2x_{i+1}v_{i+2}$;
\item $[v_i,v_{i+1}]=-x_i v_{i+2}$;
\item $[v_i,[v_i,v_{i+1}]]=[v_i^2,v_{i+1}]=-v_{i+2}$;
\item $[v_i,v_{i+2}]=2x_i x_{i+1}x_{i+3} v_{i+4}$.
\end{enumerate}
\end{Lemma}

\begin{proof}
To prove~(i) we use~\eqref{pivot2} and~\eqref{square2}:
\begin{eqnarray*}
v_i^2=( \dd_i+x_i x_{i+1}v_{i+2})^2
 = \dd_i^2+(x_ix_{i+1}v_{i+2})^2+[\dd_i,x_i x_{i+1}v_{i+2}]=x_{i+1}v_{i+2}.
\end{eqnarray*}
Now (ii) follows by $[v_i,v_i]=2v_i^2$.
The remaining claims are checked below
\begin{align*}
[{v_i,v_{i+1}}]
& =[\dd_i+x_i x_{i+1}v_{i+2},\dd_{i+1}+x_{i+1}x_{i+2}v_{i+3}]
  =[x_i x_{i+1}v_{i+2},\dd_{i+1}]=-x_i v_{i+2};\\
 [{v_i^2,v_{i+1}}]
 &=[v_i,[v_i,v_{i+1}]]=[\dd_i+x_i x_{i+1}v_{i+2},-x_i v_{i+2}]=-v_{i+2};\\
[v_i,v_{i+2}]
 &=[\dd_i+x_i x_{i+1}v_{i+2},v_{i+2}]=x_i x_{i+1}[v_{i+2},v_{i+2}]=2x_i x_{i+1}x_{i+3}v_{i+4}.
\qedhere
\end{align*}
\end{proof}

Consider general products of the pivot elements.
\begin{Lemma} \label{Lprodutos}
For any integers $i,k\geq 0$, we have:
\begin{enumerate}
\item if $k$ is even, then $[v_i,v_{i+k}]=2x_i\cdots x_{i+k-1}x_{i+k+1}v_{i+k+2}$;
\item if $k$ is odd, then $[v_i,v_{i+k}]=-x_i\cdots x_{i+k-1}v_{i+k+1}$.
\end{enumerate}
\end{Lemma}

\begin{proof}
Let $k$ be even, we use item (ii) of Lemma~\ref{Lbasicas}, and expanded presentation~\eqref{pivot3}:
\begin{align*}
[v_i,v_{i+k}]
 &=[\dd_i+x_ix_{i+1}\dd_{i+2}+\ldots+x_i\cdots x_{i+k-3}\dd_{i+k-2}+  x_i\cdots x_{i+k-1}v_{i+k},v_{i+k}]\\
 &=x_i\cdots x_{i+k-1}[v_{i+k},v_{i+k}]= 2x_i\cdots x_{i+k-1}x_{i+k+1}v_{i+k+2}.
\end{align*}
Let $k$ be odd. We use item (iii) of Lemma~\ref{Lbasicas}
\begin{align*}
[v_i,v_{i+k}]
&=[\dd_i+x_ix_{i+1}\dd_{i+2}+\ldots+x_i\cdots x_{i+k-4}\dd_{i+k-3}+x_i\cdots x_{i+k-2}v_{i+k-1},v_{i+k}]\\
&=x_i\cdots x_{i+k-2}[v_{i+k-1},v_{i+k}]=-x_i\cdots x_{i+k-1}v_{i+k+1}.\qedhere
\end{align*}
\end{proof}

Consider Lie superalgebras generated by two consecutive pivot elements: $L_i=\Lie(v_i,v_{i+1})$, $i\geq 0$,
in particular, $L_0=\RR$.
\begin{Lemma}\label{Lpivot}
Consider the Lie superalgebra $\RR=\Lie(v_0,v_1)$. We have
\begin{enumerate}
\item $v_i\in\RR$, $i\geq 0$ (moreover, we get these elements using Lie bracket only in case of arbitrary $K$);
\item $\tau^i:\RR\rightarrow L_i$ is an isomorphism for all $i\geq 0$;
\item $\RR$ is infinite dimensional.
\end{enumerate}
\end{Lemma}
\begin{proof}
We use claim (iv) of Lemma~\ref{Lbasicas}.
\end{proof}

The notion of self-similarity plays an important role in group theory~\cite{Grigorchuk00horizons,Nekr05}.
The Fibonacci Lie algebra is "self-similar"~\cite{PeSh09} but not in terms of
the definition of self-similarity given by Bartholdi~\cite{Bartholdi15}.
Namely, a Lie algebra $L$ is called {\it self-similar} if it affords a homomorphism~\cite{Bartholdi15}:
$$\psi:L\rightarrow\Der R\rightthreetimes (R\otimes L),$$
where $R$ is a commutative algebra, $\Der R$ its Lie algebra of derivations naturally acting on $R$.
This definition easily extends to Lie superalgebras, namely,
one considers $R$ to be a supercommutative associative superalgebra.
For example, a self-similarity embedding should look like:
$$\RR\hookrightarrow\langle\dd_0\rangle_K\rightthreetimes\big(\Lambda(x_0)\otimes \tau(\RR)\big).$$
Unlike two examples of self-similar Lie superalgebras~\cite{Pe16} and a family that includes many
nil self-similar restricted Lie algebras~\cite{Pe17},
it seems that such a self-similarity embedding for $\RR$ does not exist.

\section{Monomial basis of Lie superalgebra $\RR$}

Now our goal is to describe a clear monomial
basis for the (restricted) Lie superalgebra $\RR=\Lie(v_0,v_1)$ for arbitrary field~$K$.
Let us introduce a notation widely used below.
By $r_n$ denote a {\it tail} monomial:
\begin{eqnarray}\label{cauda}
r_n=x_0^{\xi_0}\cdots x_n^{\xi_n}\in \Lambda, \quad \xi_i\in\{ 0,1\};\quad n\geq 0.
\end{eqnarray}
For $n<0$ we assume that $r_n=1$. If needed, other monomials of type \eqref{cauda} will be denoted by $r'_n,r''_n$, etc.

\begin{Theorem}\label{Tbase}
Let $\ch K\neq 2$. A basis of the Lie superalgebra $\RR=\Lie(v_0,v_1)$ is given by
\begin{enumerate}
\item standard monomials of the first type:
$$ \{ r_{n-2}v_n\mid n\geq 0\};$$
\item standard monomials of the second type:
$$ \{ r_{n-3}x_{n-1}v_n\mid n\geq 2\}\setminus\{ x_0x_2v_3\}. $$
\end{enumerate}
\end{Theorem}
\begin{proof}
Let us call $n$ a {\it length}, $v_n$ a {\it head}, $r_{n-2}$  ($r_{n-3}$, sometimes, $r_{n-3}x_{n-1}$) a {\it tail},
and the optional letter $x_{n-1}$ a {\it neck} of the monomial.
We call $x_0x_2v_3$ a {\em false monomial}.

First, we prove that our monomials belong to $\RR$.
We show by induction on length $n$ that all monomials of the first type belong to $\RR$.
Clearly, such monomials of lengths $n=0,1$, namely, $v_0,v_1$, belong to $\RR$.
Consider $n\geq 2$, by inductive assumption,
all elements $r_{k-2}v_k$ with $0\leq k\leq n-1$ belong to $\RR$. Consider
\begin{align*}
[v_{n-2}^2,r_{n-3}v_{n-1}]&=r_{n-3}[v_{n-2}^2, v_{n-1}]=-r_{n-3}v_n,\\
[v_{n-2},r_{n-3}v_{n-1}]&=\pm r_{n-3}[v_{n-2},v_{n-1}]=\mp r_{n-3}x_{n-2}v_n.
\end{align*}
We conclude that all monomials of the first type of length $n$ belong to $\RR$.

Consider monomials of the second type.
By $[v_0,v_0]=2x_1 v_2$, $[v_1,v_1]=2x_2 v_3$, we obtain all such monomials of lengths 2,3.
Let $n\ge 4$. Depending on parity of $n$, using Lemma~\ref{Lprodutos}, we get
\begin{align*}
[v_0,v_{n-2}]& =2x_0\cdots x_{n-3}x_{n-1}v_n\in\RR, \hskip3.4cm  n\ \text{even};\\
[v_1,x_0 v_{n-2}]&=-x_0[v_1,v_{n-2}]=-2x_0x_1\cdots x_{n-3}x_{n-1}v_n\in \RR, \quad  n\ \text{odd}.
\end{align*}
Multiplying these elements by $v_i$, where $0\leq i\leq n-3$, we can delete any factors $x_i$ above.
We conclude that all monomials of the second type of length $n$ belong to $\RR$.
\medskip

Second, we prove that commutators of the standard monomials are expressed via the standard monomials.
Some efforts below are necessary to verify that the false monomial $x_0x_2v_3$ shall not appear.
Sometimes, we write standard monomials as $r_{n-1}v_n$.

A) Consider products of standard monomials of the same length $n$.
\begin{eqnarray*}
[r_{n-1}v_n,r'_{n-1}v_n]=\pm 2r''_{n-1}x_{n+1}v_{n+2},
\end{eqnarray*}
a standard monomial of the second type, except eventually for $n=1$.
But the only possible such product is $[v_1,v_1]=2x_2v_3$,
and the false monomial is not possible.
The false monomial cannot appear here and below by weight arguments, see below Lemma~\ref{monomiofalso}.

B) Consider two standard monomials of different lengths $u=r_{n-1}v_n$, $v=r'_{m-1}v_m$, where $n<m$.

B1) Assume that $n\equiv m({\rm mod}\ 2)$. We use presentation~\eqref{pivot3}:
\begin{align}
[u,v]&=[r_{n-1}v_n,r'_{m-1}v_m]\nonumber\\
&=\Big[r_{n-1}\Big(\dd_n+x_n x_{n+1}\dd_{n+2}+\ldots
  +x_n \cdots x_{m-3}\dd_{m-2}
  +x_n \cdots x_{m-1}v_m\Big),r'_{m-1}v_m\Big]\nonumber\\
\label{mmmm}
&=r_{n-1}\Big(\dd_n(r'_{m-1})+x_n x_{n+1}\dd_{n+2}(r'_{m-1}) +\ldots
  +x_n\cdots x_{m-3}\dd_{m-2}(r'_{m-1})\Big)v_m\\
&\quad \pm r_{n-1}x_n \cdots x_{m-1}r'_{m-1}[v_m,v_m].
\nonumber
\end{align}
The last summand equals $\pm 2r''_{m-1}x_{m+1}v_{m+2}$, which is of the second type.
Assume that $v$ is of the first type, namely, $v=r'_{m-1}v_m=r''_{m-2}v_m$,
then all monomials in~\eqref{mmmm} are of the first type.
Assume that $v$ is of the second type, namely, $v=r'_{m-1}v_m=r''_{m-3} x_{m-1}v_m$,
then all summands~\eqref{mmmm} are monomials of the second type
because the derivatives cannot remove the letter $x_{m-1}$ while the remaining factors
yield $\tilde r_{m-3}$.

B2) Assume that $n\not\equiv m ({\rm mod}\ 2)$. Again, we use~\eqref{pivot3}:
\begin{align*}
[u,v]&=[r_{n-1}v_n,r'_{m-1}v_m]\\
&=\Big[r_{n-1}\Big(\dd_n+x_n x_{n+1}\dd_{n+2}+\ldots
  +x_n \cdots x_{m-4}\dd_{m-3}
  +x_n \cdots x_{m-2}v_{m-1}\Big),r'_{m-1}v_m\Big]\\
&=r_{n-1}\Big(\dd_n(r'_{m-1})+x_n x_{n+1}\dd_{n+2}(r'_{m-1}) +\ldots
  +x_n\cdots x_{m-4}\dd_{m-3}(r'_{m-1})\Big)v_m\\
&\quad \pm r_{n-1}x_n \cdots x_{m-2}r'_{m-1}[v_{m-1},v_m].
\end{align*}
The last summand equals $\mp \tilde r_{m-1}x_{m-1}v_{m+1}$, which is of the first type.
The remaining monomials have the same type as that of $v$ by the arguments above.
%
\end{proof}

\begin{Corollary}\label{C_Lie_subring}
Let $\ch K=0$.
Consider $\Lie_\Z(v_0,v_1)$, the Lie super-subring in $\RR=\Lie(v_0,v_1)$,
generated by $v_0,v_1$ using coefficients in $\Z$.
It has the following $\Z$-basis:
\begin{enumerate}
\item standard monomials of the first type;
\item  $\lbrace x_{n-1}v_n\mid n\geq 2\rbrace$ (squares of the pivot elements);
\item $\{2r_{n-3}x_{n-1}v_n\mid r_{n-3}\ne 1,\ r_{n-3}x_{n-1}v_n\ne x_0x_2v_3,\ n\geq 2\}$
($2$-multiples of the remaining standard monomials of the second type).
\end{enumerate}
\end{Corollary}
\begin{proof}
The result follows from computations above.
The axioms of a Lie-super ring include the formal square (Section~\ref{Sdef}) and yield the second claim.
\end{proof}
In case $p=2$ superderivatives are ordinary derivatives and our example yields also a Lie algebra.
\begin{Corollary}\label{Cchar2}
Let $\ch K=p=2$.
\begin{enumerate}
\item
A basis of the Lie algebra $\RR=\Lie(v_0,v_1)$ is given by monomials of the first type.
\item
A basis of the Lie superalgebra $\RR=\Lie(v_0,v_1)$,
as well as a basis of the restricted Lie (super)algebra $\Lie_p(v_0,v_1)$
is given by
\begin{enumerate}
\item standard monomials of the first type;
\item  $\lbrace x_{n-1}v_n\mid n\geq 2\rbrace$ (squares of the pivot elements).
\end{enumerate}
\end{enumerate}
\end{Corollary}
\begin{proof}
Let $L=L_{\bar{0}}\oplus L_{\bar{1}}\subset \WW(\Lambda)$ be the Lie subalgebra generated by $v_0,v_1$
(i.e. we use the bracket $\left[\ ,\ \right]$ only).
By proof of Theorem~\ref{Tbase}, its basis $\lbrace w_j\mid j\in J\rbrace$
consists of monomials of the first type, thus yielding the first claim.

In order to obtain the restricted Lie algebra
$\Lie_p(v_0,v_1)$, which coincides with the $p$-hull $\Lie_p(L)$,
it is sufficient to add all $p^n$-powers of basis elements:
$\Lie_p(L)=\langle w_j^{[p^n]}\mid n\geq 1,j\in J\rangle_K$~\cite{StrFar}.
These powers are trivial except squares of the pivot elements.
The same observation holds in case of the restricted Lie superalgebra.

Consider the case of the Lie superalgebra $\RR=\Lie(v_0,v_1)$.
Now, we need to add squares of a basis of the odd component $L_{\bar 1}$.
Again, we add the same squares of the pivot elements.
\end{proof}
\begin{Corollary}\label{Cchar_p}
Let $\ch K=p>2$.
A basis of the restricted Lie superalgebra $\Lie_p(v_0,v_1)$
is the same as that for the Lie superalgebra $\RR$ described in Theorem~\ref{Tbase}.
\end{Corollary}
\begin{proof}
Powers of the even basis elements of $\Lie(v_0,v_1)$ are trivial because they contain Grassmann letters.
\end{proof}
\section{Weight functions, $\mathbb{Z}^2$-gradation, and two coordinate systems}
\label{Sweight}

In this section we introduce two weight functions.
Using them we establish that our algebras are $\Z^2$-graded my multidegree in the generators and
derive further corollaries.

We start with the Lie superalgebra $\WW(\Lambda_I)$
of special superderivations of the Grassmann algebra
$\Lambda_I=\Lambda(x_i\mid i\in I)$
and consider a subalgebra spanned by {\it pure} Lie monomials:
$$\WW_{\mathrm{fin}}(\Lambda_I)= \langle x_{i_1}\cdots x_{i_m}\dd_j
\mid i_k,j\in I\rangle_K\subset \WW(\Lambda_I).$$
Define a {\it weight function} on the Grassmann variables and respective superderivatives related as:
$$wt(\dd_i)=-wt(x_i)=\alpha_i\in\C,\qquad i\geq 0,$$
and extend it to pure Lie monomials
$wt(x_{i_1}\cdots x_{i_m}\dd_j)=-\a_{i_1}-\cdots-\a_{i_m}+\a_j$, $i_k,j\in I$.
One checks that the weight function is {\it additive}, namely,
$wt([w_1,w_2])=wt(w_1)+wt(w_2)$, where $w_1,w_2$ are pure Lie monomials.
The weight function is also extended to an associative hull $\Alg(\WW_{\mathrm{fin}}(\Lambda_I))$
and it is additive on associative products of its monomials.

Now we return to $\RR=\Lie(v_0,v_1)$ and $\AA=\Alg(v_0,v_1)$.
We want all terms in~\eqref{pivot2} to have the same weight. Namely, assume that
$wt(v_i)=wt(\dd_i)=\alpha_{i}=-\alpha_{i}-\alpha_{i+1}+\alpha_{i+2}$, $i\ge 0$.
We get a recurrence relation
\begin{equation}\label{recorrencia}
\alpha_{i+2}=\alpha_{i+1}+2\alpha_{i},\qquad i\geq 0.
\end{equation}
The roots of the respective characteristic polynomial are $\lambda_1=2$ and $\lambda_2=-1$.
\begin{Lemma} \label{Lpesos}
A base of the space of solutions of the recurrence equation~\eqref{recorrencia} is given by
two weight functions $\wt(\ )$ and $\swt(\ )$ defined as follows:
\begin{enumerate}
\item $\wt(\dd_n)=\wt(v_n)=-\wt(x_n)=2^n$, $n\geq 0$ (the {\rm weight} function);
\item $\swt(\dd_n)=\swt(v_n)=-\swt(x_n)=(-1)^n$, $n\geq 0$ (the {\rm superweight} function);
\item we combine these functions together as
$\Wt(v_n)=\left(\wt(v_n),\swt(v_n)\right)=\left(2^n,(-1)^n\right),\ n\geq 0$,
(the {\rm vector weight} function).
\end{enumerate}
\end{Lemma}

Below, a {\it monomial} is
any (Lie or associative) product of letters
$\lbrace x_i,\dd_i,v_i\mid i\geq 0\rbrace\subset \End\Lambda$.
By arguments above, we have the following.
\begin{Lemma}\label{aditividade}
The weight functions are well defined on monomials.
They are additive on (Lie or associative) products of monomials, e.g.,
$\Wt(a\cdot b)=\Wt(a)+\Wt(b)$, where $a,b$ are monomials of $\AA$.
\end{Lemma}

As a first application we establish $\Z^2$-gradations.
\begin{Lemma}\label{Lz2graduacao}
The algebras $\RR=\Lie(v_0,v_1)$ and $\AA=\Alg(v_0,v_1)$ are $\mathbb{Z}^2$-graded
by multidegree in $\lbrace v_0,v_1\rbrace$:
$$\RR=\mathop{\oplus}\limits_{n_1,n_2\geq 0}\RR_{n_1, n_2},\qquad
  \AA=\mathop{\oplus}\limits_{n_1,n_2\geq 0}\AA_{n_1, n_2}.$$
\end{Lemma}
\begin{proof}
Consider the weight vectors for the generators of our algebras
$\Wt(v_0)=(1,1)$ and $\Wt(v_1)=(2,-1)$ given by Lemma~\ref{Lpesos}.
For $n_1,n_2\geq 0$, let $\RR_{n_1 n_2}\subset\RR$ be the subspace spanned
by all Lie elements of multidegree $(n_1,n_2)$ in $\{v_0,v_1\}$.
By Lemma~\ref{aditividade}, all elements $v\in\RR_{n_1 n_2}$ have the same vector weight:
\begin{equation}\label{vetor peso}
\Wt(v)=n_1\Wt(v_0)+n_2\Wt(v_1)=(n_1+2n_2,n_1-n_2).
\end{equation}
Elements of $\RR_{n_1 n_2}\subset\WW(\Lambda_I)$
are written as (probably infinite) linear combinations of pure Lie monomials.
Since the vectors $\Wt(v_0)$, $\Wt(v_1)$ are linearly independent,
two different components $\RR_{n_1, n_2}$, $\RR_{n_1', n_2'}$  have different weights,
hence they are expressed via different sets of pure Lie monomials.
We conclude that the sum of the components  is direct.
The $\Z^2$-gradation follows by definition of these components.
\end{proof}

Let $v\in\AA_{n_1 n_2}$, then $\swt(v)=n_1\swt(v_0)+n_2\swt(v_1)=n_1-n_2$.
Now, $\swt(v)=0$ if and only if $n_1=n_2$.
We call $\mathop{\oplus}\limits_{n\geq 0}\AA_{nn}$ a {\it diagonal}
of the $\mathbb{Z}^2$-gradation of $\AA$.
Similarly, $\mathop{\oplus}\limits_{n\geq 0}\RR_{nn}$ is
a {\it diagonal} of the $\mathbb{Z}^2$-gradation of $\RR$.
The additivity of the superweight function easily implies the following fact.
\begin{Corollary}\label{Ctriang}
The following statements are valid:
\begin{enumerate}
\item  The algebras $\RR$ and $\AA$ allow triangular decompositions
into direct sums of three subalgebras:
$$\RR=\RR_+ \oplus \RR_0 \oplus \RR_-,\qquad \AA=\AA_+ \oplus \AA_0 \oplus \AA_-,$$
where the respective components are spanned by monomials of positive, zero,
and negative superweights.
\item  The zero components above coincide with the respective diagonals:
$$\RR_0=\mathop{\oplus}\limits_{n\geq 0}\RR_{nn},\qquad
\AA_0=\mathop{\oplus}\limits_{n\geq 0}\AA_{nn}.$$
\end{enumerate}
\end{Corollary}

Given a nonzero homogeneous element
$v\in \AA_{n_1 n_2}$, $n_1,n_2\geq 0$, we define its {\it multidegree} vector and  {\it degree}:
$$\Gr(v)=(n_1,n_2)\in\mathbb{Z}^2\subset\mathbb{R}^2,\qquad \deg(v)=n_1+n_2.$$
We put it on plane using {\it standard coordinates} $(X_1,X_2)\in\mathbb{R}^2$,
which we also call {\it multidegree coordinates}.
Thus, we write $\Gr(v)=(n_1,n_2)=(X_1,X_2)$.

Let $(X_1,X_2)\in\mathbb{R}^2$ be an arbitrary point of plane in terms of the standard coordinates.
We introduce its {\it weight coordinates} $(Z_1,Z_2)$:
\begin{equation}\label{relacao}
\begin{cases}
Z_1=X_1+2X_2;\\
Z_2=X_1-X_2;
\end{cases}
\qquad\qquad
\begin{cases}
X_1=(Z_1+2Z_2)/3;\\
X_2=(Z_1-Z_2)/3.
\end{cases}
\end{equation}

\begin{Lemma}
Let $v\in\AA$ be a monomial, $\Gr(v)=(n_1,n_2)=(X_1,X_2)$ its multidegree, and $(Z_1,Z_2)$
the respective weight coordinates. Then
$$(Z_1,Z_2)=\Wt(v)=(\wt(v),\swt(v)).$$
\end{Lemma}

\begin{proof}
By~\eqref{vetor peso} and~\eqref{relacao}, we have
\begin{align*}
\wt(v)&=n_1+2n_2=X_1+2X_2=Z_1,\\
\swt(v)&=n_1-n_2=X_1-X_2=Z_2.\qedhere
\end{align*}
\end{proof}

\begin{Lemma}\label{multigraupivo}
Consider the pivot elements $\{v_n\mid n\ge 0\}$. Then
\begin{enumerate}
\item $\Gr(v_n)=\frac{1}{3}\left(2^n+2(-1)^n,2^n-(-1)^n\right)$,\quad $n\geq 0;$
\item The pivot elements belong to two parallel lines: $X_1-X_2=1$ and $X_1-X_2=-1$.
\end{enumerate}
\end{Lemma}

\begin{proof}
By Lemma~\ref{Lpesos}, $\Wt(v_n)=\left(2^n,(-1)^n\right)=(Z_1,Z_2)$,
we use~\eqref{relacao} and obtain the first claim.
Observe that  $X_1-X_2=Z_2=\swt(v_n)=(-1)^n$.
\end{proof}

\begin{Lemma}\label{monomiofalso}
Consider the false monomial $u=x_0x_2v_3$.
\begin{enumerate}
\item $\Gr(u)=(-1,2)$;
\item  the false monomial cannot appear as a homogeneous component of an element of $\RR$.
\end{enumerate}
\end{Lemma}
\begin{proof} We use Lemma~\ref{Lpesos} and~\eqref{relacao}
\begin{align*}
\Wt(u)&=(Z_1,Z_2)=\Wt(v_3)-\Wt(v_2)-\Wt(v_0)=(8,-1)-(4,1)-(1,1)=(3,-3);\\
\Gr(u)&=(Z_1+2Z_2,Z_1-Z_2)/3=(-1,2).
\end{align*}
Any homogeneous element of $\RR=\Lie(v_0,v_1)$
must have nonnegative multidegree coordinates in the generators.
Therefore, $v=x_0x_2v_3$ cannot appear.
\end{proof}

\begin{Lemma}\label{Lend}
Let $0\neq v\in\AA_{n_1n_2}$, $n_1,n_2\geq0$, and $\tau(v)$ the image under the endomorphism $\tau$.
Then
$$ \Gr(\tau(v))=(2n_2,n_1+n_2). $$
\end{Lemma}
\begin{proof}
The relation $[v_0^2,v_1]=-v_2$ implies that $\Gr(v_2)=(2,1)$.
By assumption,
$v$ is a linear combination of products involving $n_1$ factors $v_0$ and $n_2$ factors $v_1$.
Since $\tau$ is an endomorphism,
$\tau(v)$ is a linear combination of products involving $n_1$ factors $\tau(v_0)$ and $n_2$ factors $\tau(v_1)$.
Using $\tau(v_0)=v_1$, $\tau(v_1)=v_2$, and additivity of the multidegree function, we get
\begin{equation*}
\Gr(\tau(v))=n_1\Gr(v_1)+n_2\Gr(v_2)=n_1(0,1)+n_2(2,1)=(2n_2,n_1+n_2).
\qedhere
\end{equation*}
\end{proof}
The multidegree $\Z^2$-gradation induces a {\it weight $\N$-gradation}:
$$
\RR=\mathop{\oplus}\limits_{n=1}^\infty\tilde \RR_n,\quad
\tilde \RR_n=\langle w\in \RR_{n_1,n_2}\mid \wt w=n_1+2n_2=n\rangle_K,\quad n\ge 1.
$$
Similarly, one defines a {\it superweight $\Z$-gradation}:
$$
\RR=\mathop{\oplus}\limits_{n=-\infty}^{+\infty}\RR_n',\quad
\RR_n'=\langle w\in \RR_{n_1,n_2}\mid \swt w=n_1-n_2=n\rangle_K,\quad n\in \Z.
$$
Finally, we have a {\it degree $\N$-gradation}:
$$
\RR=\mathop{\oplus}\limits_{n=1}^\infty\RR_n,\quad
\RR_n=\langle w\in \RR_{n_1,n_2}\mid \deg w=n_1+n_2=n\rangle_K,\quad n\ge 1.
$$
The last gradation is important to us because it is related to the lower central series.
\begin{Lemma} \label{Llowercentral}
The terms of the lower central series of $\RR$ are as follows:
\begin{enumerate}
\item $\RR^m=\mathop{\oplus}\limits_{n\ge m} \RR_n$, $m\ge 1$;
\item the lower central series factors are isomorphic to the terms of the degree $\N$-gradation:
      $\RR^{n}/\RR^{n+1}\cong  \RR_n$, for all $n\ge 1$;
\item let $\ch K\ne 2$, then $\dim\RR/\RR^m\approx 3m$, as $m\to\infty$;
\item let $\ch K= 2$, then $\dim\RR/\RR^m\approx 3m/2$, as $m\to\infty$.
\end{enumerate}
\end{Lemma}
\begin{proof}
Recall that the multidegree component $\RR_{n_1 n_2}$ is spanned by commutators containing $n_1$ factors $v_0$ and $n_2$ factors $v_1$.
On the other hand, the term of the lower central series $\RR^m$, $m\ge 1$, is spanned
by all commutators of length at least $m$ in the generators $\{v_0,v_1\}$~\cite{Ba}, thus
$\RR^m=\mathop{\oplus}\limits_{n_1+n_2\ge m}\RR_{n_1n_2}=\mathop{\oplus}\limits_{n\ge m}\RR_n$.

By Theorem~\ref{TRgrowth}, using the ordinary growth function
we have $\dim \RR/\RR^m=\dim(\RR_1\oplus\cdots \oplus \RR_{m-1})= \gamma_{\RR}(m-1)\approx 3m$, $m\to \infty$ in the case $\ch K\ne 2$.
\end{proof}
\begin{Corollary}
$\RR$ is a {\em naturally graded algebra} with respect to
the degree $\N$-gradation in the sense~\cite{Mil},
namely, the associated graded algebra related to the filtration by the lower central series
$(\RR^m\mid m\ge 1)$
is isomorphic to the original $\N$-graded algebra $\RR=\mathop{\oplus}\limits_{n=1}^\infty\RR_n$.
\end{Corollary}
\section{Bounds on weights of monomials of  Lie superalgebra $\RR$}\label{Sweightbounds}

In this section we establish estimates on weights and superweights of
the standard monomials of the Lie superalgebra $\RR$,
which allow us to prove that the standard monomials are situated in a region of plane restricted by two
logarithmic curves (Theorem~\ref{TcurvesL}).
\begin{Lemma}\label{Lestimativas}
Weights of standard monomials of the first and second type satisfy the following inequalities:
\begin{align*}
2^{n-1} &<\wt(r_{n-2}v_n)\leq 2^n,\qquad\qquad n\geq 0;\\
2^{n-2}&<\wt(r_{n-3}x_{n-1}v_n)\leq 2^{n-1},\quad n\geq 2.
\end{align*}
\end{Lemma}
\begin{proof}
First, observe that weight of a tail $r_m=x_0^{\xi_0}\cdots x_m^{\xi_m},\ \xi_i\in\lbrace 0,1\rbrace,$ has the following bounds:
\begin{equation}\label{desigualdades}
0\geq\wt(r_m)\geq -(2^0+2^1+\cdots +2^m)=-2^{m+1}+1,\quad m\geq 0.
\end{equation}
Theses bounds are also formally valid for $m=-1$. Using these bounds we obtain:
$$ 2^n\geq\wt(r_{n-2}v_n)\geq -2^{n-1}+1+2^n=2^{n-1}+1>2^{n-1},\quad n\geq1. $$
The outside bounds above are also valid for $n=0$.
Let us check bounds for monomials of the second type:
\begin{equation*}
2^{n-1}=2^n-2^{n-1}\geq\wt(r_{n-3}x_{n-1}v_n)\geq -2^{n-2}+1-2^{n-1}+2^n=2^{n-2}+1>2^{n-2},\quad n\ge 2.
\qedhere
\end{equation*}
\end{proof}

\begin{Lemma}\label{superpeso}
Superweights of standard monomials have the following bounds:
\begin{align*}
-\frac{n}{2}-\frac{1}{2} &\leq\swt(r_{n-2}v_n)\leq \frac{n}{2},\qquad\qquad\ n\geq1;\\
-\frac{n}{2}-\frac{3}{2} &\leq\swt(r_{n-3}x_{n-1}v_n)\leq \frac{n}{2}+1,\quad n\geq2.
\end{align*}
\end{Lemma}
\begin{proof}
Let $w=r_{n-2}v_n$, $n\geq 1$.
Consider the case $n=2k$, $k\ge 1$.
We have $r_{n-2}=x_0^{\xi_0}x_1^{\xi_1}\cdots x_{2k-3}^{\xi_{2k-3}}x_{2k-2}^{\xi_{2k-2}}$, $0\leq\xi_i\leq1$,
where
$\swt(x_0)=\swt(x_2)=\cdots = \swt(x_{2k-2})=-1$,
$\swt(x_1)=\swt(x_3)=\cdots = \swt(x_{2k-3})=1$, and
$\swt(v_n)=1$.
We get:
\begin{align*}
-k &\leq \swt(r_{n-2})\leq k-1;\\
-\frac{n}{2}+1=-k+1  &\leq \swt(r_{n-2}v_n)\leq k=\frac{n}{2}.
\end{align*}

Consider the case $n=2k+1$, $k\ge 0$.
Similarly, we get
\begin{align*}
-k &\leq\swt(r_{n-2})\leq k;\\
-\frac{n}{2}-\frac{1}{2}=-k-1 &\leq \swt(r_{n-2}v_n)\leq k-1=\frac{n}{2}-\frac{3}{2}.
\end{align*}
These cases yield the claimed bounds for monomials of the first type.

Now consider a standard monomial of the second type $w=r_{n-3}x_{n-1}v_n$, $n\geq 2$.
In case $n=2k$, $k\ge 1$,
we have $r_{n-3}=x_0^{\xi_0}x_1^{\xi_1}\cdots x_{2k-4}^{\xi_{2k-4}}x_{2k-3}^{\xi_{2k-3}}$, $0\leq\xi_i\leq1$,
and $\swt(x_{2k-1}v_{2k})=2$. We get
\begin{align*}
-k+1 &\leq {\rm swt}(r_{n-3})\leq k-1;\\
-\frac{n}{2}+3=-k+3 &\leq\swt(r_{n-3}x_{n-1}v_n)\leq k+1=\frac{n}{2}+1.
\end{align*}
Consider the case $n=2k+1$, $k\ge 1$. Similarly, we get
\begin{align*}
-k &\leq {\rm swt}(r_{n-3})\leq k-1;\\
-\frac{n}{2}-\frac{3}{2}=-k-2 &\leq\swt(r_{n-3}x_{n-1}v_n)\leq k-3=\frac{n}{2}-\frac{7}{2}.
\end{align*}
These two estimates yield the desired bounds for monomials of the second type.
\end{proof}

We combine two previous lemmas along with a formal check in case $n=0$ for $\swt(v_0)=1$.
\begin{Lemma}\label{gerais}
Let $w$ be a standard monomial of length $n\geq 0$. Then
\begin{align*}
2^{n-2} &<\wt(w)\leq 2^n,\\
-\frac{n}{2}-\frac{3}{2} &\leq\swt(w)\leq\frac{n}{2}+1.
\end{align*}
\end{Lemma}

\begin{Theorem}
\label{TcurvesL}
The points of plane associated with the standard monomials of $\RR=\Lie(v_0,v_1)$
are bounded by two logarithmic curves in terms of the weight coordinates $\Wt(w)=(Z_1,Z_2)$:
$$-\frac{1}{2}\log_2Z_1-\frac{5}{2}<Z_2<\frac{1}{2}\log_2 Z_1+2.$$
\end{Theorem}

\begin{proof}
Let $w$ be a standard monomial of length $n\geq 0$.
By Lemma~\ref{gerais}, we have $2^{n-2}<\wt(w)=Z_1$, thus $n<\log_2Z_1+2$.
By the same lemma, we get bounds on $Z_2=\swt(w)$:
\begin{equation*}
-\frac{1}{2}\log_2Z_1-\frac{5}{2}<-\frac{n}{2}-\frac{3}{2}\leq Z_2\leq\frac{n}{2}+1<\frac{1}{2}\log_2 Z_1+2.
\qedhere
\end{equation*}
\end{proof}

\section{Structure of homogeneous components and growth of Lie superalgebra $\RR$}
\label{SmonomialsR}

In this section we prove that the $\Z^2$-grading of $\RR$ is fine (Theorem~\ref{Tfine}).
We show that the growth of $\RR=\Lie(v_0,v_1)$ is linear, moreover,
we explicitly compute the weight growth function $\tilde{\gamma}_{\RR}(m)$
and determine an asymptotic for the ordinary growth function $\gamma_{\RR}(m)$ (Theorem~\ref{TRgrowth}).

First, we specify standard monomials of fixed weight.
\begin{Lemma}\label{Ldoismonomios}
Suppose $\ch K\neq2$ and fix an integer $m\ge 1$.
\begin{enumerate}
\item There is exactly one standard monomial of the first type $w_1$ with $\wt(w_1)=m$ (valid for $\ch K=2$).
\item In the remaining claims assume that $m\notin\{1,3\}$.
      There is exactly one standard monomial of the second type $w_2$ with $\wt(w_2)=m$.
\item There are exactly two standard monomials of weight $m$;
      (in case $m\in\{1,3\}$, there is exactly one standard monomial of weight $m$).
\item Let $w_1=r_{n-2}v_n$ be as above, then $w_2=r_{n-2}x_nv_{n+1}$.
\item $\swt(w_1)-\swt(w_2)=3\cdot(-1)^n$.
\item $\Gr (w_1)-\Gr (w_2)=(-1)^n(2,-1)$.
\end{enumerate}
\end{Lemma}
\begin{proof}
We refine consideration of weights in Lemma~\ref{Lestimativas}.
Fix $n\ge 1$. Consider all tails
$r_{n-2}=x_0^{\xi_0}\cdots x_{n-2}^{\xi_{n-2}}$, $\xi_i\in\{0,1\}$, and respective weights
$\wt(r_{n-2})=-\sum_{i=0}^{n-2}\xi_i 2^i$.
By properties of the $p$-adic expansion of integers, there is a one-to-one correspondence
between the set of these tails and
the list of their weights, which fills an interval of integers $\{-2^{n-1}{+}1,\ldots,-1,0 \}$ without gaps.
Adding $\wt(v_n)=2^n$, we obtain a one-to-one correspondence between
the set of standard monomials of the first type
$r_{n-2}v_n$ and the list of their weights $\{2^{n-1}{+}1,\ldots,2^n\}$, where $n\ge 1$ is fixed.
We take unions of these sets
and obtain a one-to-one correspondence between the set
of standard monomials of the first type of length $n\ge1$ and the list of their weights $m\ge 2$.
It remains to consider the unique monomial of zero length $v_0$, $\wt(v_0)=1$.
Claim~(i) is proved.

Fix $n\ge 2$. By the arguments above, there is
a one-to-one correspondence between the set of tails $r_{n-3}$  and
the list of their weights $\{-2^{n-2}{+}1,\ldots,-1,0 \}$ without gaps.
Adding $\wt(x_{n-1}v_n)=-2^{n-1}+2^n=2^{n-1}$,
we obtain a one-to-one correspondence between
the set of standard monomials of the second type
$r_{n-3}x_{n-1}v_n$ and the list of their weights $\{2^{n-2}{+}1,\ldots,2^{n-1}\}$, where $n\ge 2$ is fixed.
By taking unions of these sets, we obtain a one-to-one correspondence between
the set of standard monomials of the second type of length $n\ge2$ and their weights $m\ge 2$.
It remains to exclude the false monomial of weight $\wt(x_0x_2v_3)=3$.

Claim~(iii) follows from (i) and (ii).

To prove (iv), assume that $w_1=r_{n-2}v_n$
is a standard monomial of the first type with $\wt(w_1)=m\neq 1,3$.
Clearly, $w_2=r_{n-2}x_nv_{n+1}$ is a standard monomial of the second type. It has the same weight $m$:
$$
\wt(w_2)-\wt(w_1)=\wt(r_{n-2}x_nv_{n+1})-\wt(r_{n-2}v_n)=\wt(x_n)+\wt(v_{n+1})-\wt(v_n)=-2^n+2^{n+1}-2^n=0.
$$

Using Lemma~\ref{Lpesos}, we prove claim~(v):
$$ \swt(w_1)-\swt(w_2)  = \swt(r_{n-2}v_n)-\swt(r_{n-2}x_nv_{n+1})
= \swt(v_n)-\swt(x_n)-\swt(v_{n+1})=3\cdot(-1)^n.$$
Finally, applying Lemma~\ref{multigraupivo}, we check claim (vi):
\begin{align*}
&\Gr(w_1)-\Gr(w_2)  =  \Gr(r_{n-2}v_n)-\Gr(r_{n-2}x_nv_{n+1})
 =  2\Gr(v_n)-\Gr(v_{n+1})\\
& =  \frac{2}{3}\left(2^n+2(-1)^n,2^n-(-1)^n\right)
  -\frac{1}{3}\left(2^{n+1}+2(-1)^{n+1},2^{n+1}-(-1)^{n+1}\right) =  (-1)^n(2,-1).
\qedhere
\end{align*}
\end{proof}

\begin{Corollary}\label{Clines}
Fix $m\in \N$, $m\notin\{1,3\}$.
Let $(X_1,X_2)$ and $(Z_1,Z_2)$ be the standard  and weight coordinates on plane.
The line of fixed weight $Z_1=X_1+2X_2=m$
contains exactly two standard monomials whose position differ by vector $(2,-1)$.
In case $m\in\{1,3\}$, the line contains one standard monomial.
(see Fig.~\ref{Fig1}).
\end{Corollary}
\begin{Corollary}\label{Cnaturgrad}
Consider the weight $\N$-gradation
$\RR=\mathop{\oplus}\limits_{n=1}^\infty\tilde \RR_n$
(Section~\ref{Sweight}).
Then $\dim\tilde \RR_n=2$ for $n\ge 1$, except $\dim \tilde \RR_1=\dim \tilde \RR_3=1$.
\end{Corollary}

Next, we easily derive that the multidegree $\mathbb{Z}^2$-gradation of $\RR=\Lie(v_0,v_1)$ is  fine.
\begin{Theorem}\label{Tfine}
Components of the $\mathbb{Z}^2$-gradation
$\RR=\mathop{\oplus}\limits_{n_1,n_2\geq 0}\RR_{n_1, n_2}$
by multidegree in the generators $\lbrace v_0,v_1\rbrace$
(Lemma~\ref{Lz2graduacao}) are at most one-dimensional.
\end{Theorem}
\begin{proof}
Consider a lattice point $(n_1, n_2)\in\Z^2\subset \R^2$.
It belongs to a line of fixed weight $Z_1=X_1+2X_2=n_1+2n_2$,
that contains at most two standard monomials that differ by $(2,-1)$ (Corollary~\ref{Clines}).
\end{proof}

\begin{Theorem}\label{TRgrowth}
The growth of the Lie superalgebra $\RR$ is linear with the following properties.
\begin{enumerate}
\item Let $\ch K\neq2$. We compute the weight growth function explicitly:
$$\tilde{\gamma}_{\RR}(m)=2m-2,\qquad m\geq3.$$
\item Let $\ch K=2$. We compute the weight growth function explicitly:
$$\tilde{\gamma}_{\RR}(m)=m+[\log_2m],\qquad m\geq1.$$
\item $\GKdim\RR=\LGKdim\RR=1$.
\item Let $\ch K\neq 2$. The ordinary growth function satisfies:
$\gamma_{\RR}(m)\approx 3m$, $m\to \infty$.
\item Let $\ch K=2$. The ordinary growth function satisfies:
$\gamma_{\RR}(m)\approx 3m/2$, $m\to \infty$.
\end{enumerate}
\end{Theorem}
\begin{proof}
Consider $\ch K\neq 2$. Using claim~(iii) of Lemma~\ref{Ldoismonomios},
we compute the number of standard monomials of weight not exceeding $m$ and obtain that
$\tilde{\gamma}_{\RR}(m)=2m-2$, $m\ge 3$.

Let $\ch K=2$.
By Corollary $\ref{Cchar2}$, standard monomials of the first type and squares of the pivot elements,
namely $\lbrace x_{n-1}v_n\mid n\geq 2\rbrace$, constitute a basis of $\RR$.
For each $m\geq 1$ we have exactly one standard monomial of the first
type of weight $m$ (Lemma~\ref{Ldoismonomios}).
One has $\wt(x_{n-1}v_n)=2^{n-1}$, $n\ge 2$.
Now, claim (ii) follows. In any characteristic, we get $\GKdim\RR=\LGKdim\RR=1$.

Let us prove claims (iv) and (v).
Fix notations. Let $v$ be a standard monomial, $(X_1,X_2)$, $(Z_1,Z_2)$
its multidegree and weight coordinates, and $n=\deg(v)=X_1+X_2$ its degree.
Below, $m\in\N$ will be a parameter.
The ordinary growth function $\gamma_{\RR}(m)$, $m\in\N$,
counts standard monomials $v\in\RR_{X_1X_2}$ of bounded degree, i.e.
$n=X_1+X_2\leq m$.
By~\eqref{relacao}, we have
$$n=X_1+X_2=\frac{Z_1+2Z_2}{3}+\frac{Z_1-Z_2}{3}=\frac{2}{3}Z_1+\frac{1}{3}Z_2.$$
We apply estimates of Theorem~\ref{TcurvesL}
\begin{align*}
\frac{2}{3}Z_1-\frac{\log_2 Z_1+5}{6}
&\leq n\leq
\frac{2}{3}Z_1+\frac{\log_2 Z_1+4}{6};\\
\lim_{Z_1\to +\infty}\frac {n}{Z_1}&=\frac 23;\\
\lim_{n\to +\infty}\frac{Z_1} {n}&=\frac 32.
\end{align*}
Hence, for any $\delta>0$ there exists $N_\delta\in\mathbb{N}$ such that
\begin{equation}\label{epsilon}
\left(\frac{3}{2}-\delta\right)n\leq Z_1\leq \left(\frac{3}{2}+\delta\right)n,\qquad
n\geq N_\delta.
\end{equation}
Let $M_{\delta}$ be the number of standard monomials of degree less than $N_\delta$.
Fix $m\in\mathbb{N}$.
Consider a standard monomial $v$ of degree $n$ such that $N_{\delta}\le n\le m$.
By upper bound~\eqref{epsilon},
\begin{equation}\label{cotaparaz1}
Z_1\leq \left(\frac{3}{2}+\delta\right)n\leq\left(\frac{3}{2}+\delta\right)m.
\end{equation}
Now assume that $\ch K\neq 2$.
For each weight $Z_1$ there are at most two standard monomials (Lemma~\ref{Ldoismonomios}).
Using~\eqref{cotaparaz1}, we get an upper bound on the growth:
\begin{equation}\label{cotasuperiorgama}
\gamma_{\RR}(m)\leq 2\left(\frac{3}{2}+\delta\right)m+M_{\delta},\quad m\ge 1.
\end{equation}
Consider standard monomials with bounded weight:
\begin{equation}\label{Z_1}
Z_1\le \left(\frac{3}{2}-\delta\right)m.
\end{equation}
By lower inequality~\eqref{epsilon} it follows that $n\le m$ for these standard monomials
of degree $n\ge N_{\delta}$.
For each weight $Z_1$ satisfying~\eqref{Z_1} we have two standard monomials
(actually, we shall use (i)) and their degree is bounded by $m$.
We also subtract the number $M_\delta$ of monomials which might not satisfy~\eqref{epsilon}.
So, we get a lower bound:
\begin{equation}\label{cotainferiorgama}
\gamma_{\RR}(m)\geq 2\left(\frac{3}{2}-\delta\right)m-2-M_\delta,\quad m\ge 1.
\end{equation}
Since $\delta>0$ was chosen arbitrary, \eqref{cotasuperiorgama} and \eqref{cotainferiorgama}
imply that $\mathop{\lim}\limits_{m\rightarrow\infty} \gamma_{\RR}(m)/m=3$.

Consider $\ch K=2$. Similarly, using item (ii) of Theorem~\ref{TRgrowth}, we get
$$\left(\frac{3}{2}-\delta\right)m-M_\delta
\leq\gamma_{\RR}(m)\leq \left(\frac{3}{2}+\delta\right)m+\left[\log_2((3/2+\delta)m)\right]+M_{\delta},\quad m\ge 1.$$
We conclude that $\mathop{\lim}\limits_{m\rightarrow\infty} \gamma_{\RR}(m)/m=3/2$.
\end{proof}

\section{Basis, bounds on weights, and growth of associative hull $\AA$}
\label{SweightboundsA}

Now our goal is to determine growth of the associative hull $\AA=\Alg(v_0,v_1)$.
For a series of previous examples of (fractal, or self-similar) (restricted) Lie (super)algebras, bases for
respective associative hulls were not found~\cite{PeSh09,PeShZe10,PeSh13fib,Pe17,Pe16}.
Instead, we used a bypass approach.
We considered bigger (restricted) Lie (super)algebras $\tilde{\RR}\supset\RR$
whose  bases were given by so called {\it quasi-standard monomials}
and we determined and used bases of their associative hulls $\tilde{\AA}=\Alg(\tilde{\RR})\supset\AA$.

The virtue of the present example is that we are able to describe explicitly a basis of the associative hull~$\AA$.
In this section we find a basis of $\AA$ and establish bounds on its weights.
These bounds allow us to describe a geometric position of the basis of $\AA$ on plane.
Finally, we specify the growth of $\AA$.

\begin{Theorem}\label{TbasisA}
Let $\AA=\Alg(v_0,v_1)\subset\End(\Lambda)$, $\ch K\ne 2$. Then
\begin{enumerate}
\item a basis of $\AA$ is given by monomials
\begin{align*}
&\{v_0, v_1v_0^*\}\bigcup \left\{x_0^{\xi_0}x_1^{\xi_1} v_2v_1^*v_0^*
\ \Big|\ \xi_0+\xi_1\le 1\right\}\bigcup\\
&\quad\bigcup_{n=3}^\infty
  \left\{ r_{n-3}x_{n-2}^{\xi_{n-2}}x_{n-1}^{\xi_{n-1}}v_n v_{n-1}^{\alpha_{n-1}}v_{n-2}^{\alpha_{n-2}}\cdots v_0^{\alpha_0}
\ \Big|\
\xi_i,\alpha_i\in\{ 0,1\};\ \xi_{n-2}+\xi_{n-1}\leq 1+\alpha_{n-1}\right\},
\end{align*}
where $r_{n-3}$ are tail monomials, and $b^*$ denote all $b^\a$, $\a\in\{0,1\}$.
In case $n=3$, we additionally assume that the
monomials containing both $\{x_0,x_2\}$ are only of type:
$ x_0x_2v_3v_2v_1^*v_0^*;$
\item let $\ch K=2$, then $\AA$ is contained in span of the monomials above;
\item monomials with $\xi_{n-1}=0$ are linearly independent and belong to $\AA$ in case of any field.
\end{enumerate}
\end{Theorem}

\begin{proof}
We refer to $n$ above as to the {\it length} of the respective monomial.
First, we check that the listed monomials belong to $\AA$.
Let $n\ge 4$.
Consider
$w=r_{n-3}x_{n-2}^{\xi_{n-2}}x_{n-1}^{\xi_{n-1}}v_n v_{n-1}^{\alpha_{n-1}}v_{n-2}^{\alpha_{n-2}}\cdots v_0^{\alpha_0}$,
satisfying our condition $\xi_{n-2}+\xi_{n-1}\leq 1+\alpha_{n-1}$.
We have two cases a) If $\xi_{n-2}+\xi_{n-1}\leq 1$, then
the initial factor $r_{n-3}x_{n-2}^{\xi_{n-2}}x_{n-1}^{\xi_{n-1}}v_n$
is a standard monomial.
Since $v_i\in\RR$, we conclude that $w\in\AA$.
b) Consider the case $\xi_{n-2}=\xi_{n-1}=1$, by the assumed inequality, $\alpha_{n-1}=1$.
Then also
\begin{equation*}
w=(r_{n-3}x_{n-1}v_n)\cdot(x_{n-2}v_{n-1})\cdot v_{n-2}^{\alpha_{n-2}}\cdots v_0^{\alpha_0}\in\AA.
\end{equation*}
The cases $n=0,1,2$ are easily checked. Consider $n=3$.
The arguments in two cases above fail only when they lead to using the false monomial $x_0x_2v_3$
as the initial factor.
Thus, the algorithm above fail to get monomials containing both $\{x_0,x_2\}$.
We have to check that we can obtain the monomials specified in the wording that contain both $\{x_0,x_2\}$.
To obtain the claimed monomials we use that $x_0x_2v_3v_2=x_2v_3\cdot x_0v_2\in \AA$.
The linear independence follows by the same arguments as~\cite[proof of Theorem 4.1]{PeSh09}.

Second, we shall prove that products of the standard monomials are expressed via the associative monomials listed above.
We work with products of standard monomials, we reorder such products using PBW-like arguments,
where a total order is fixed obeying to the length $n$ of standard monomials.
In this process, we shall eliminate not only squares of odd standard monomials but also products of two
standard monomials of the same length:
$$(r_{n-1}v_n)\cdot(r'_{n-1}v_n)=\pm r''_{n-1}v_n^2=\pm r''_{n-1}x_{n+1}v_{n+2}.$$
The obtained monomial is standard
(we cannot obtain the false monomial because one has the only unsuccessful possibility to get it: $v_1\cdot v_1=x_2v_3$).
As a result, we get products of standard monomials of different lengths, written in strictly length-decreasing order.
For example, we obtain:
\begin{equation}\label{decrescente}
r_{n-1}v_n\cdot(r_{n-2}v_{n-1})\cdots (v_0),
\end{equation}
where we have one standard monomial of length $n$, while the remaining
standard monomials of lengths $n-1,\ldots,0$ are optional.

Consider $n\ge 4$.
Now, we move all Grassmann letters in (\ref{decrescente}) to the left.
Let $x_i$ be a Grassmann variable in a standard monomial $r_{j-1}v_j$, $j<n$, then  $i<j$.
The standard monomials before it in (\ref{decrescente}) have lengths greater than $j$, thus, greater than $i$.
So, $x_i$ supercommutes with all preceding heads $v_k$.
We obtain an associative monomial as desired,
we only need to check the inequality on the senior indices.
Consider two cases. a) The factor $r_{n-2}v_{n-1}$ does not appear in (\ref{decrescente}).
The senior factor $r_{n-1}v_n$ is standard, thus,
it contains at most one of the variables $\lbrace x_{n-2},x_{n-1}\rbrace$.
Since it is not possible to get any of these variables from monomials of smaller length,
we get $\xi_{n-2}+\xi_{n-1}\leq 1$ and $\alpha_{n-1}=0$, and the required inequality is valid.
b)~The factor $r_{n-2}v_{n-1}$ appears in (\ref{decrescente}).
Then we get $\alpha_{n-1}=1$ and the desired inequality on the senior indices is satisfied.

Consider $n=2$. Recall that all factors in~\eqref{decrescente} are standard monomials.
So, the first factor is either $x_0v_2$ or $x_1v_2$, followed by optional factors $v_1$, $v_0$.
We get the desired monomials of length 2.
The cases $n=0,1$ are trivial.

Consider $n=3$.
We can get $x_2$ in~\eqref{decrescente} only from the first factor.
By Theorem~\ref{Tbase}, there is only one such standard monomial, namely, $r_2v_3=x_2v_3$.
Now, $x_0$ can appear only from the second factor, and there is only one such standard monomial: $r_1v_2=x_0v_2$.
The possible remaining factors are $v_1$ and $v_0$.
Thus, we arrive at the claimed monomials: $x_0x_2v_3v_2v_1^*v_0^*$.
The first claim is proved.

The second claim follows by the same arguments, some monomials can only disappear.
The third claim follows because $r_{n-2}v_n,v_{n-1},\ldots,v_0\in\RR\subset\AA$.
\end{proof}

\begin{Lemma}\label{estimativasatil}
Let $w=r_{n-3}x_{n-2}^{\xi_{n-2}}x_{n-1}^{\xi_{n-1}}v_n v_{n-1}^{\alpha_{n-1}}v_{n-2}^{\alpha_{n-2}}\cdots v_0^{\alpha_0}$
be a basis monomial of $\AA$ of length $n$.
Then
$$2^{n-2}<\wt(w)<2^{n+1},\qquad |\swt(w)|\leq n+1;\qquad n\ge 0.$$
\end{Lemma}
\begin{proof}
For $n\in\lbrace 0,1,2\rbrace$, these estimates are verified directly. Suppose that $n\geq 3$.
Let $w$ contain at most one of Grassmann variables $\lbrace x_{n-2},x_{n-1}\rbrace$, we have
$$2^{n-2}<2^{n-2}+1=-(2^0+\cdots+2^{n-3})-2^{n-1}+2^n \leq\wt(w)\leq 2^n+2^{n-1}+\cdots+2^0=2^{n+1}-1<2^{n+1}.$$
In case $\xi_{n-2}=\xi_{n-1}=1$,
$w$ must contain the factor $v_{n-1}$.
So, we have the desired estimates again:
$$2^{n-2}<2^{n-1}+1=-(2^0+\cdots+2^{n-1})+2^n+2^{n-1}\leq\wt(w)\leq 2^n+2^{n-1}+\cdots+2^0=2^{n+1}-1<2^{n+1}.$$

Consider the superweight. Let $n=2k$.
By estimates $-k\leq\swt(r_{n-1})\leq k$, $\swt(v_{2k})=1$, and
$-k\leq\swt(v_{2k-1}^{\alpha_{2k-1}}\cdots v_0^{\alpha_0})\leq k$,
we have
$$-2k+1\leq\swt(w)\leq 2k+1.$$
Let $n=2k+1$. We have $-k-1\leq\swt(r_{n-1})\leq k$, $\swt(v_{2k+1})=-1$,
$-k\leq\swt(v_{2k}^{\alpha_{2k}}\cdots v_0^{\alpha_0})\leq k+1$, thus,
$$-2k-2\leq\swt(w)\leq 2k.$$
Both cases yield $\vert\swt(w)\vert\leq n+1$.
\end{proof}

\begin{Theorem}\label{TcurvesA}
Points of plane depicting the basis monomials of $\AA=\Alg(v_0,v_1)$
are bounded by two logarithmic curves in terms of the weight coordinates $\Wt(w)=(Z_1,Z_2)$:
$$\vert Z_2\vert<{\rm log}_2 Z_1+3.$$
\end{Theorem}
\begin{proof}
Let $w$ be a basis monomial of $\AA$ of length $n$.
By Lemma \ref{estimativasatil}, one has estimates:
$2^{n-2}<\wt(w)=Z_1$,  and $|Z_2|=|\swt(w)|\leq n+1$.
Thus,  $\vert Z_2\vert<{\rm log}_2 Z_1+3$.
\end{proof}

\begin{Theorem}\label{TgrowthA}
Let $\AA=\Alg(v_0,v_1)$. The growth of $\AA$ is quadratic, namely:
\begin{enumerate}
\item there exist constants $c_1,c_2,c_3,c_4>0$ such that
\begin{enumerate}
\item $c_1m^2\leq\tilde{\gamma}_{\AA}(m)\leq c_2m^2$, $m\ge 1$ (the weight growth function);
\item $c_3m^2\leq\gamma_{\AA}(m)\leq c_4m^2$, $m\ge 1$ (the ordinary growth function);
\end{enumerate}
\item $\GKdim\AA=\LGKdim\AA=2$.
\end{enumerate}
\end{Theorem}
\begin{proof}
First, let us establish bounds on the weight growth function.
Let $m\geq 2$ be fixed, set $n=[\log_2 m]-1$.
Consider all monomials $w=r_{n-2}v_n v_{n-1}^{\alpha_{n-1}}v_{n-2}^{\alpha_{n-2}}\cdots v_0^{\alpha_0}\in\AA$
(i.e. $\xi_{n-1}=0$, Theorem~\ref{TbasisA}, (iii)).
By Lemma~\ref{estimativasatil}, $\wt(w)<2^{n+1}\leq m$.
The number of these monomials yields a desired lower bound on the weight growth function:
\begin{equation*}
\tilde{\gamma}_{\AA}(m)\geq 2^{n-1}\cdot 2^n=2^{2n-1}\geq 2^{2\log_2 m-5}=\frac{1}{32}m^2.
\end{equation*}

Fix $m\ge 1$ and set $n=[\log_2 m]+3$.
Consider a monomial
$w=r_{j-3}x_{j-2}^{\xi_{j-2}}x_{j-1}^{\xi_{j-1}}
v_j v_{j-1}^{\alpha_{j-1}}\cdots v_0^{\alpha_0}\in\AA$
(without inequality on indices of Theorem~\ref{TbasisA}).
If $j>n$, then using Lemma~\ref{estimativasatil} $\wt(w)>2^{j-2}>2^{n-2}>m$.
Thus, consider monomials $w$ of bounded weight: $\wt(w)\leq m$,
then for all of them we have $j\leq n$.
Therefore, the number of monomials $w$ of length at most $n$
yields an upper bound on the weight growth function $\tilde{\gamma}_{\AA}(m)$.
Let $j\geq 0$.
We have $2^{j}$ possibilities for the Grassmann variables
and $2^{j}$ possibilities for $\alpha_{j-1},\ldots,\alpha_0$.
So, we get a desired upper bound on the weight growth function:
\begin{equation*}
\tilde{\gamma}_{\AA}(m)
\le\sum_{j=0}^n 2^{2j}
<\frac{4^{n+1}}{3}
\le\frac{4^{\log_2m+4}}{3}= \frac{2^8}{3} m^2.
\end{equation*}

We have $1\le \wt(v_0),\wt(v_1)\le 2$.
One observes that
$\tilde\gamma_\AA(m)\le \gamma_\AA(m)\le \tilde\gamma_\AA( 2m)$, $m\ge 1$.
Now one obtains bounds on the ordinary growth function. Part (ii) is trivial.
\end{proof}
\begin{Remark}
Probably, one can find the weight growth function of $\AA$ explicitly similar to Theorem~\ref{TRgrowth}.
\end{Remark}

\begin{figure}[h]
\caption{One-dimensional components $\RR_{n,m}$ are shown by dots.
Green and blue points are standard monomials of the first and second types, respectively.
Dashed red arrows point to pivot elements and
red lines depict points with weights equal to that of pivot elements.
Two green logarithmic curves restrict positions of standard monomials.
Grey dashed lines show points with constant degrees.}
\label{Fig1}
\begin{center}
\includegraphics[width=1.0\textwidth]{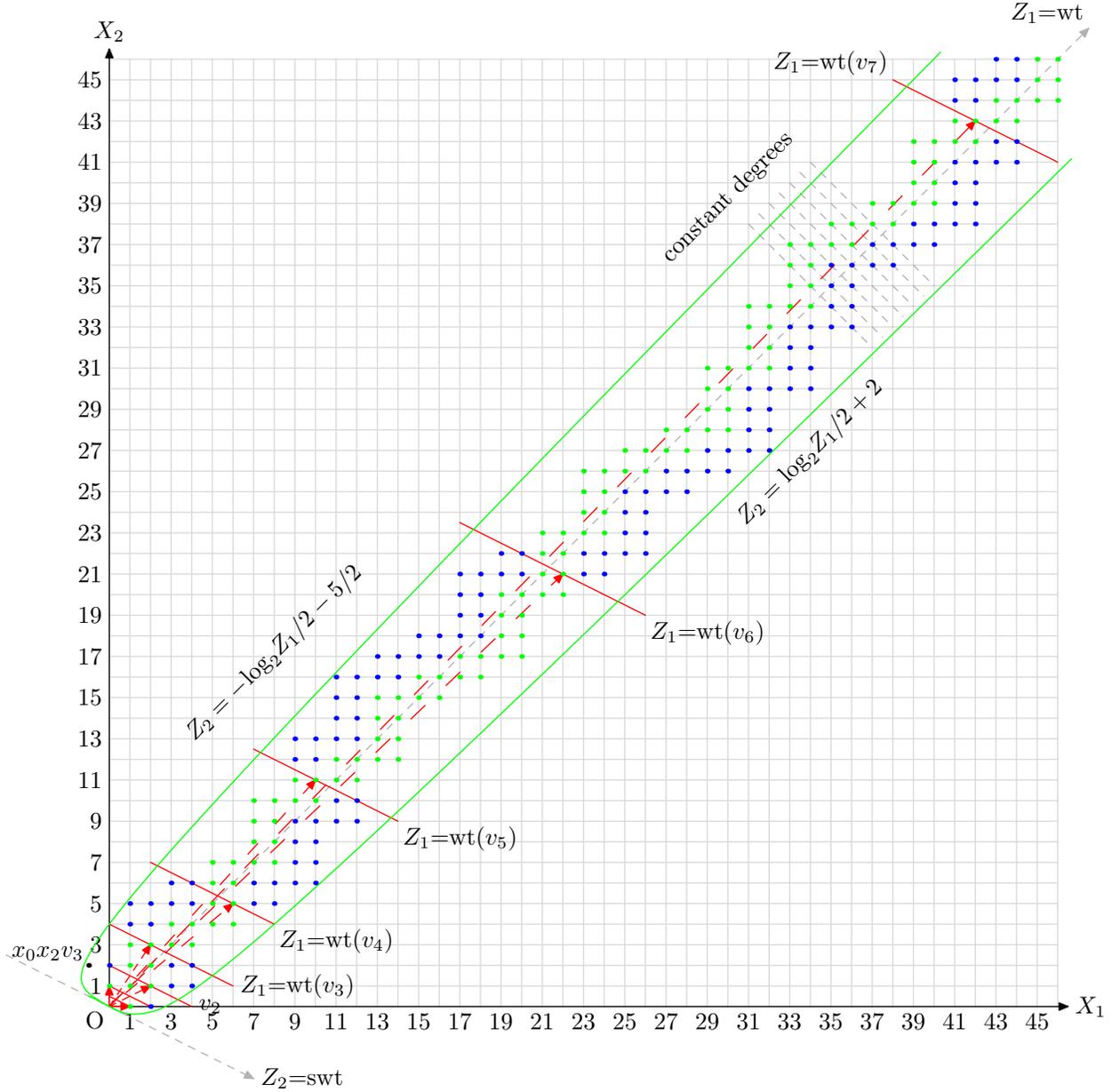}
\end{center}
\end{figure}

\section{Generating functions of Lie superalgebra $\RR$}\label{Sfunctions}

Recall that components of the $\mathbb{Z}^2$-gradation of $\RR$ are at most one-dimensional (Theorem~\ref{Tfine}).
In this section we study this gradation in terms of generating functions.
We establish recursive relations.
The recursive relations allows us to compute an initial part of the generating function,
which simplifies the proof that $\RR$ is of finite width in the next section.

Let $A=\!\!\!\mathop{\oplus}\limits_{n,n\in\mathbb{Z}}\!\! A_{nm}$ be a $\mathbb{Z}^2$-graded algebra and
$A=\!\! \mathop{\oplus}\limits_{n\in\mathbb{Z}}\!\!  A_n$,
where $A_n=\!\!\! \mathop{\oplus}\limits_{m+k=n}\!\! A_{mk}$,
a induced $\mathbb{Z}$-gradation.
Define respective {\it generating functions} (or {\it Hilbert functions}):
\begin{align*}
\mathcal{H}(A,t_1,t_2)&=\sum_{n,m}\dim A_{nm}t_1^nt_2^m;\\
\mathcal{H}(A,t)&=\sum_n \dim A_nt^n=\mathcal{H}(A,t,t).
\end{align*}
Similarly, we define generating functions for graded sets and spaces.
Recall that the Lie superalgebra $\RR=\Lie(v_0,v_1)$ is
$\mathbb{Z}^2$-graded by multidegree in the generators (Lemma~\ref{Lz2graduacao}).
Denote by $T_n^k$ the set of standard monomials of length $n$ and type $k$,
by $T_n$ all standard monomials of length $n$, where $n\ge 0$, $k\in\lbrace1,2\rbrace$.
Let $T$ be the set of all standard monomials, it is a homogeneous
basis of $\RR=\Lie(v_0,v_1)$ (in case $\ch K\ne 2$).

\begin{Lemma}\label{Ltau}
Let $U\subset T$ be a subset of standard monomials and $\tau: \RR\to \RR$ the shift endomorphism. Then
$$ \H(\tau(U),t_1,t_2)=\H(U, t_2,t_1^2t_2). $$
\end{Lemma}
\begin{proof}
Consider $a\in U_{n_1,n_2}$. By Lemma~\ref{Lend}, $\Gr(\tau(a))=(2n_2,n_1+n_2)$.
Hence,
\begin{equation*}
\H(\tau(U_{n_1,n_2}),t_1,t_2)=\dim U_{n_1,n_2}t_1^{2n_2}t_2^{n_1+n_2}
=\dim U_{n_1,n_2}t_2^{n_1}(t_1^2t_2)^{n_2}
=\H(U_{n_1,n_2},t_2,t_1^2t_2).
\qedhere
\end{equation*}
\end{proof}
\begin{Lemma}\label{Lbijecao}
Let $\ch K\ne 2$. For all lengths $n\ge 4$ we have bijections:
\begin{enumerate}
\item $T_{n+1}^k=\lbrace 1,x_0\rbrace\cdot\tau(T_n^k)$ for $k\in\{0,1\}$,
      (in case $k=1$, valid also for $n\ge 1$);
\item $T_{n+1}=\lbrace 1,x_0\rbrace\cdot\tau(T_n)$;
\item and equalities:\
$\H(T_{n+1},t_1,t_2)=(1+t_1^{-1})\H(\tau(T_n),t_1,t_2)= (1+t_1^{-1})\H(T_n,t_2,t_1^2t_2).$
\end{enumerate}
\end{Lemma}
\begin{proof}
The first and second claim follow directly by the structure of the standard monomials (Theorem~\ref{Tbase}).
Since $\Gr(x_0)=-\Gr(v_0)=(-1,0)$, we get
$\mathcal{H}(\lbrace 1,x_0\rbrace,t_1,t_2)=1+t_1^{-1}.$
It remains to apply Lemma~\ref{Ltau}.
\end{proof}

\begin{Lemma}
Let $\RR=\Lie(v_0,v_1)$, $\ch K\neq2$.
Then  $\mathcal{H}(\RR,t_1,t_2)$ satisfies a  recursive relation:
$$
\mathcal{H}(\RR,t_1,t_2) =  t_1+t_1^2+t_1^3t_2+t_1^4t_2-t_1^{-1}t_2-t_1^{-1}t_2^2+(1+t_1^{-1})\cdot\mathcal{H}({\bf R},t_2,t_1^2t_2).
$$
\end{Lemma}
\begin{proof}
We have
\begin{align*}
T_0 & =  \lbrace v_0\rbrace;\qquad
T_1   =  \lbrace v_1\rbrace;\qquad
T_2   =  \lbrace v_2,x_0v_2,x_1v_2\rbrace;\\
T_3 & =  \lbrace v_3,x_0v_3,x_1v_3,x_2v_3,x_0x_1v_3\rbrace;\\
T_4 & =  \lbrace v_4,x_0v_4,x_1v_4,x_2v_4,x_3v_4,x_0x_1v_4,x_0x_2v_4,
        x_0x_3v_4,x_1x_2v_4,x_1x_3v_4,x_0x_1x_2v_4,x_0x_1x_3v_4\rbrace.
\end{align*}
We find the respective generating functions:
\begin{align*}
\mathcal{H}(T_0,t_1,t_2) & =  t_1;\qquad
\mathcal{H}(T_1,t_1,t_2)   =  t_2;\qquad
\mathcal{H}(T_2,t_1,t_2)   =  t_1^2+t_1t_2+t_1^2t_2;\\
\mathcal{H}(T_3,t_1,t_2) & =  t_2^2+t_1t_2^2+t_1^2t_2^2+t_1t_2^3+t_1^2t_2^3;\\
\mathcal{H}(T_4,t_1,t_2) & =  t_1^3t_2+t_1^4t_2+t_1^3t_2^2+t_1^4t_2^2+t_1^3t_2^3+t_1^4t_2^3
     +t_1^3t_2^4+t_1^4t_2^4+t_1^5t_2^4+t_1^6t_2^4+t_1^5t_2^5+t_1^6t_2^5.
\end{align*}
By Lemma \ref{Lbijecao}, we have the following bijection:
\begin{equation}\label{t0t1}
T\setminus\left(T_0\cup T_1\cup T_2\cup T_3\cup T_4\right)=\lbrace 1,x_0\rbrace\cdot\tau\left(T\setminus\left(T_0\cup T_1\cup T_2\cup T_3\right)\right).
\end{equation}

Using~\eqref{t0t1} and Lemma~\ref{Ltau}, we establish the claimed relation:
\begin{align*}
\mathcal{H}(\RR,t_1,t_2) & =  \mathcal{H}(T_0\cup T_1\cup T_2\cup T_3\cup T_4,t_1,t_2)
   +(1+t_1^{-1})\cdot\mathcal{H}(\tau(T\setminus(T_0\cup T_1\cup T_2\cup T_3)),t_1,t_2)\\
& =  \mathcal{H}(T_0\cup T_1\cup T_2\cup T_3\cup T_4,t_1,t_2)\\
&\quad  +(1+t_1^{-1})\cdot\Big(\mathcal{H}({\bf R},t_1,t_2)
   -\mathcal{H}(T_0\cup T_1\cup T_2\cup T_3,t_1,t_2)\Big)\Big|_{t_1:=t_2,\ t_2:=t_1^2t_2}\\
& =  t_1+t_2+t_1^2+t_1t_2+t_2^2+t_1^2t_2+t_1t_2^2+t_1^3t_2+t_1^2t_2^2+t_1t_2^3+t_1^4t_2+t_1^3t_2^2\\
&\quad  +t_1^2t_2^3+t_1^4t_2^2+t_1^3t_2^3+t_1^4t_2^3+t_1^3t_2^4+t_1^4t_2^4+t_1^5t_2^4+t_1^6t_2^4+t_1^5t_2^5+t_1^6t_2^5\\
&\quad  +(1+t_1^{-1})\cdot\mathcal{H}({\bf R},t_2,t_1^2t_2)-t_2-t_1^2t_2-t_1^2t_2^2-t_2^2-t_1^2t_2^3-t_1^4t_2^2-t_1^4t_2^4\\
&\quad  -t_1^6t_2^5-t_1^6t_2^4-t_1^4t_2^3-t_1^{-1}t_2-t_1t_2-t_1t_2^2-t_1^{-1}t_2^2-t_1t_2^3-t_1^3t_2^2-t_1^3t_2^4
  -t_1^5t_2^5-t_1^5t_2^4-t_1^3t_2^3\\
& =  t_1+t_1^2+t_1^3t_2+t_1^4t_2-t_1^{-1}t_2-t_1^{-1}t_2^2+(1+t_1^{-1})\cdot\mathcal{H}(\RR,t_2,t_1^2t_2).
\qedhere
\end{align*}
\end{proof}
\begin{Remark}
On Fig.~\ref{Fig1} all diagonal components are occupied. But in general, this observation is not true.
We give a list of indices of initial empty diagonal components $n$ (i.e. $\dim\RR_{n,n}=0$):
\begin{align*}
&57,114,185,217,225,227,228,229,233,249,313,370,434,450,454,455,456,458,466,498,569,626,697,\\
&729,737,739,740,741,745,761,825,857,865,867,868,869,873,889,897,899,900,901,905,907,908,\\
&909,910,911,912,913,915,916,917,921,929,931,932,933,937,953,985,993,995,996,997,1001,\ldots
\end{align*}
\end{Remark}
\section{Finite width of Lie superalgebra $\RR$}\label{Swidth}

Let $\RR=\Lie(v_0,v_1)\subset\Der \Lambda$ be the Lie superalgebra as above.
Consider the degree (natural) $\N$-gradation $\RR=\mathop{\oplus}\limits_{n=1}^\infty\RR_n$,
where $\RR_n$ is spanned by basis elements of degree $n$ in the generators $\{v_0,v_1\}$, $n\ge 1$.
Consider the respective generating function
$\mathcal{H}(\RR,t)=\mathop{\sum}\limits_{n=1}^\infty\dim \RR_n t^n$.
In case $\ch K\ne 2$, a computer computation based on Lemma~\ref{Lbijecao} yields that
all coefficients are $\{2,3,4\}$, this was checked till degree 4000:
\begin{align*}
&\mathcal{H}(\RR,t)=2t+3t^2+2t^3+3t^4+4t^5+4t^6+3t^7+2t^8+3t^9+3t^{10}+2t^{11}+3t^{12}+3t^{13}+2t^{14}+3t^{15}\\
&\quad+4t^{16}+4t^{17}+3t^{18}+2t^{19}+3t^{20}+4t^{21}+4t^{22}+3t^{23}+2t^{24}+3t^{25}+4t^{26}+4t^{27}+3t^{28}+2t^{29}+3t^{30}\\
&\quad+3t^{31}+2t^{32}+3t^{33}+3t^{34}+2t^{35}+3t^{36}+4t^{37}+4t^{38}+3t^{39}+2t^{40}+3t^{41}+3t^{42}+2t^{43}+3t^{44}+3t^{45}\\
&\quad+2t^{46}+3t^{47}+4t^{48}+4t^{49}+3t^{50}+2t^{51}+3t^{52}+3t^{53}+2t^{54}+3t^{55}+3t^{56}+2t^{57}+3t^{58}+4t^{59}+4t^{60}+\ldots
\end{align*}
Let us prove that this observation is true in general.

\begin{Theorem}\label{Twidth}
Let $\RR\subset\Der\Lambda$
be an algebra  of type specified below generated by $\{v_0,v_1\}$ over a field~$K$,
$\RR=\mathop{\oplus}\limits_{n=1}^\infty \RR_n$ the degree $\N$-gradation, and
$(\RR^n\mid n\ge1)$ the lower central series.
By Lemma~\ref{Llowercentral}, we have
$a_n=\dim\RR_n=\dim \RR^n/\RR^{n+1}$, $n\ge 1$. Then
\begin{enumerate}
\item
Let $\ch K\ne 2$, $\RR=\Lie(v_0,v_1)$ the Lie superalgebra.
Then $a_n\in\{2,3,4\}$, and $\RR$ is of width 4.
\item
Let $\ch K=2$, $\RR$ the Lie algebra.
Then $a_n\in\{1,2\}$, and $\RR$ is of width 2.
\item
Let $\ch K=2$, $\RR$ the Lie superalgebra (or the restricted Lie (super)algebra).
Then $a_n\in\{1,2,3\}$, and $\RR$ is of width 3.
\end{enumerate}
\end{Theorem}
\begin{figure}[h]
\caption{Notations are the same as on Fig.~\ref{Fig1}}\label{Fig2}
\begin{center}
\includegraphics[width=1.0\textwidth]{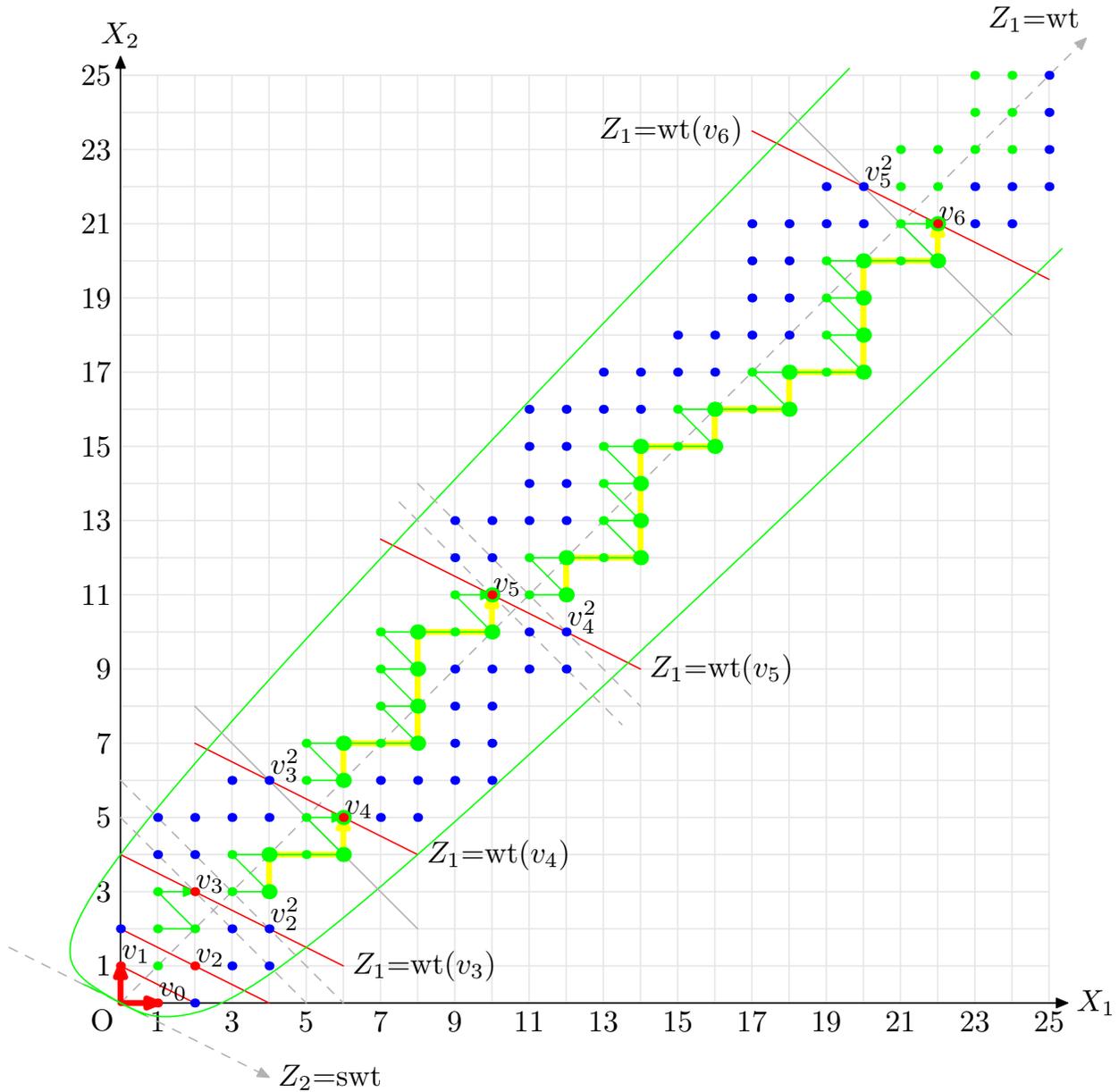}
\end{center}
\end{figure}
\begin{proof}
We start with the second claim.
By Corollary~\ref{Cchar2}, a basis of $\RR$ is given by the standard monomials of the first type.
On Fig.~\ref{Fig2} these monomials are depicted by green points except the pivot elements which are red.
By Lemma~\ref{Ldoismonomios},
for any $m\ge 1$ there is exactly one monomial of the first type and weight $m$.
By Lemma~\ref{Lestimativas}, all monomials  $w\in T_n^1$
are between two parallel red lines of fixed weight, namely,
$2^{n-1} <\wt(w)\le 2^n$, $n\ge 0$.
Thus, the sets $T_n^1$, $n\ge 0$, are in different regions of plane separated by these red lines.
We connect $T_n^1$, the set of the standard monomials of the first type of length $n$, where $n\ge 3$,
by a thin green broken line in order of their weights.
We prove and use the following geometric phenomena labeled by letters.

(A) {\it The broken line for $T_n^1$, $n\ge 3$, consists of alternating parts of two types,
namely, an oblique segment $a=(-1,1)$ is followed by either a horizontal segment $b=(1,0)$
or the three horizontal segments $b$.}
One checks that the steps $a$ and $b$, indeed, increase weights of vertices by one.
We prove (A) by induction on~$n$.
The green broken lines for $T_n^1$, $3\le n\le 6$, are shown on Fig.~\ref{Fig2}, they clearly satisfy (A).
The case $n=3$ is the base of the induction.
Arguing by induction, assume that the monomials of $T_n^1$, where $n\ge 3$ is fixed, satisfy (A).
By Lemma~\ref{Lbijecao}, we have the bijection $T_{n+1}^1=\{ 1,x_0\}\cdot\tau(T_n^1)$.
Since $\tau$ acts on plane as a liner transformation described in Lemma~\ref{Lend},
the application of $\tau$ produces a new broken line consisting of new alternating parts:
a horizontal long segment $\tilde a=\tau(a)=(2,0)$ followed by
either a vertical segment $\tilde b= \tau(b)=(0,1)$ or
the three vertical segments $\tilde b$.
These new broken lines are drawn on Fig.~\ref{Fig2} using thick yellows segments.
Now we take into account the action of two components of the factor $\{ 1,x_0\}$.
We conclude that $T_{n+1}^1$ consists of the vertices of the new broken yellow line, shown as thick points,
and their horizontal shifts by $\Gr(x_0)=(-1,0)$.
Now one really requires geometric observations, e.g.
drawing the process on a checkered paper or scrutinize it on Fig.~\ref{Fig2} to observe that
we get points belonging to a new broken green line as required, thus proving the inductive step.

(B) {\it A transformation of the broken line in the process above:
$T_n^1\rightsquigarrow T_{n+1}^1$, $n\ge 3$,
is described by a replacement rule
$\tilde{\tau}:$ $a\mapsto bb$, $b\mapsto ab$, concatenating a prefix $b$ to the result}.
One needs to repeat and refine the geometric observations above.
(The replacement rule is also formally valid for $n=1,2$).

(C){\it The broken lines for $T_n^1$, $n\ge 4$,
keep the same initial and final segments: $babbba\ldots abbbab$
(may be, with overlapping).} The cases $n=4,5,6$ are seen on Fig.~\ref{Fig2}. This property follows from (B).

(D) {\it We claim that $T_{n-1}^1$ and $T_n^1$, where $n\ge 1$,
are connected either by $b$ or $a$, for even and odd $n$, respectively.}
For $n=1$ the claim is trivial. Let $n\ge 2$.
The broken line for $T_{n-1}^1$ ends with $v_{n-1}$, while one for $T_n^1$ starts with
$w=x_0\cdots x_{n-2}v_n$, because this is the element in $T_n^1$  of smallest weight.
Assume that $n$ is even, by Lemma~\ref{Lprodutos}, $[v_0,v_{n-1}]=-x_0\cdots x_{n-2}v_n=-w$.
We get $\Wt(w)=\Wt(v_{n-1})+\Wt(v_0)=\Wt(v_{n-1})+(1,0)=\Wt(v_{n-1})+b$.
Assume that $n$ is odd, by Lemma~\ref{Lprodutos}, $[v_1,x_0v_{n-1}]=x_0\cdots x_{n-2}v_n=w$ and
$\Wt(w)=\Wt(v_{n-1})+\Wt(v_1)-\Wt(v_0)=\Wt(v_{n-1})+(-1,1)=\Wt(v_{n-1})+a$. Now (D) is proved.

(A$'$), (B$'$)
{\it Now we combine monomials $\cup_{n\ge 3}T_n^1$  in one infinite broken line.
The united broken line satisfies (A) and (B), respectively}.
By (C), we know the final part of $T_{n-1}^1$ and the initial part of $T_n^1$.
By (D), we insert between them either $b$ or $a$, thus yielding (A$'$).
By (D), consider a connecting segment $a$ (or $b$) between $T_{n-1}^1$ and $T_n^1$,
that goes to the next connecting segment between $T_{n}^1$ and $T_{n+1}^1$
which is $b$ (respectively, $a$) followed by the prefix $b$.
Thus, we obtain the same replacement rule for the letter of the connecting segment,
yielding (B$'$).

By drawing, one checks that by (A$'$) for any $k\ge 3$ the united broken line $\cup_{n\ge 3}T_n^1$
contains one or two vertices of degree $k$.
It remains to view a few points of $T_0^1\cup T_1^1\cup T_2^1$ on Fig.~\ref{Fig2} and see that
the second claim is proved.

(F) {\it The local pattern around $v_n$ moves to the local pattern around $v_{n+2}$, $n\ge 3$.}
This fact follows from (C) and (D).

Now we are ready to prove the first claim.
Let $T_n^2$ be the set
of monomials of the second type of length $n$, $n\ge 2$, which are drawn blue on Fig.~\ref{Fig2}.
By Lemma~\ref{Ldoismonomios}, monomials of the second type are obtained from ones of the first type
by shifts: $T_{n+1}^2=T_n^1+(-1)^{n}(-2,1)$, $n\ge 3$.
Using the first claim, each piece $T_{n+1}^2$, $n\ge 1$, contains one or two monomials of fixed degree.
A problem is that we cut the united green broken line into pieces
and apply different shifts $\pm(-2,1)$ to them, increasing or decreasing the degree.
Observe that (F) applies to monomials of two types as well.
Around the even pivot elements (like $v_4$, $v_6$), these shifts are in "opposite directions".
Around the odd pivot elements (like $v_3$, $v_5$),
this effect gives three blue points of fixed degree on two lines,
but the total number of monomials of fixed degree remains at most 4
(see the pairs of dashed grey lines on Fig.~\ref{Fig2}).
By (F), the same patterns are retained in neighborhoods
of the even and odd pivot elements.
Collecting green and blue points together we get $\{2,3,4\}$ monomials of fixed degree,
thus proving the first claim.

Consider the third claim.
By Corollary~\ref{Cchar2}, we take all monomials of the first type and
add squares of the pivot elements.
The squares $v_n^2=x_{n+1}v_{n+2}$, $n$ being odd,
yield totally three basis monomials on lines passing through them,
see the lines passing through $v_3^2$ and $v_5^2$ on Fig.~\ref{Fig2}.
\end{proof}

\begin{Corollary}\label{Cdiamonds}
Let $\ch K=2$, $\RR$ the Lie algebra generated by $\{v_0,v_1\}$ (i.e. without $p$-mapping),
and $a_n=\dim\RR_n=\dim \RR^n/\RR^{n+1}$, $n\ge 1$. Then
\begin{enumerate}
\item
The sequence $(a_n| n\ge 1)$, starting with $n\ge 5$, consists of alternating parts of two types:
a block $1,1$ followed either by 2 (a diamond) or by a block $2,2,2$ (a triplet of diamonds).
\item The sequence in not eventually periodic.
\end{enumerate}
\end{Corollary}
\begin{proof}
By (A$'$), $\cup_{n\ge 3}T_n^1$ is drawn using the alternating parts:
$a$ followed by either $b$ or $b^3$.
Using the replacement rule (B$'$),
$\cup_{n\ge 4}T_n^1$ consists of the alternating parts:
$bb$ followed by either $ab$ or $(ab)^3$.
By drawing the respective broken line one observes that
the last parts yield one or three diamonds, while the first part yields components $1,1$.
The initial coefficients are seen on Fig.~\ref{Fig2}, see also below.

The first part explains the origin of diamonds. The eventual periodicity
of the degree components would mean eventual periodicity of the sequence of letters $a,b$.
The sequence of monomials
contains infinitely many $v_n$, $n\ge 0$, close to the diagonal because $\swt(v_n)=\pm 1$.
An eventual periodicity would mean that
the monomials of the first type $w$ lie in a strip $|\swt(w)|\le C$.
This is not the case because
values $\swt(x_0x_2\cdots x_{2k-2}v_{2k})=1-k$, $k\ge 2$, are not bounded.
\end{proof}

\begin{Remark}
We obtain an interesting eventually non-periodic sequence
$(a_n| n\ge 1)$, which was not found in the database~\cite{Sloane}.
Below, a triplet of diamonds $2,2,2$ denote as $2^3$:
$$
2122112112^3112^3112112112^3112112112^3112^3112^3112112112^3112^3112^3112112112^3112112112^3112112\ldots\\
$$
\end{Remark}
\begin{Corollary}\label{Cnotthin}
Let $\ch K=2$ and $\RR$ the Lie algebra generated by $\{v_0,v_1\}$.
Then $\RR$ is of width 2 but it is not thin.
\end{Corollary}
\begin{proof}
By Picture~2,
$\dim \RR_{15}=\dim\RR_{16}=2$. Let $0\ne  z\in\RR_{8,7}\subset\RR_{15}$.
The covering property (Subsection~\ref{SSnarrow}) is not satisfied because
$[\RR_1, z]=\langle [v_0,z]+[v_1,z]\rangle_K\subset\RR_{9,7}+\RR_{8,8}=\RR_{8,8}
  \subsetneqq \RR_{7,9}\oplus \RR_{8,8}=\RR_{16}$.
\end{proof}
\begin{Remark}
Since $\RR$ above is not thin, the respective diamonds should be called {\em fake diamonds}.
\end{Remark}
\section{Lie superalgebra $\RR=\RR_{\bar 0}\oplus \RR_{\bar 1}$ is nil graded, non-nillity of $\AA$}

In this section we establish that the Lie superalgebra $\RR$ enjoys a property analogous to
the periodicity of the Grigorchuk and Gupta-Sidki groups.
Namely, we prove that $\RR=\RR_{\bar 0}\oplus \RR_{\bar 1}$ is  ad-nil $\Z_2$-graded (Theorem~\ref{Tadnilpotente}).
Recall that $\Z^2$-components of $\RR$ are at most one-dimensional (Theorem~\ref{Tfine}).
In particular, this implies that
an extension of Martinez-Zelmanov Theorem~\ref{TMarZel} for the Lie superalgebra
case in characteristic zero is not possible (Corollary~\ref{adnilpotentesmultigrau}).
We also prove that the components $\RR_\pm$ and $\AA_\pm$
of the triangular decompositions of Corollary~\ref{Ctriang} are locally nilpotent (Theorem~\ref{Tlocal}).

On the other hand, in case $\ch K=0$ and $\ch K=2$, we show that
the associative algebra $\AA=\Alg(v_0,v_1)$ is not nil,
the restricted Lie algebra $\RR=\Lie_p(v_0,v_1)$, $\ch K=2$,  is also not nil (Lemma~\ref{Lnaonil}).

\begin{Theorem}\label{Tadnilpotente}
Let $\RR=\Lie(v_0,v_1)=\RR_{\bar{0}}\oplus\RR_{\bar{1}}$. For any $a\in\RR_{\bar{n}}$,
$\bar{n}\in\lbrace\bar{0},\bar{1}\rbrace$, the operator $\ad(a)$ is nilpotent.
\end{Theorem}
\begin{proof}
First, assume that $a\in\RR_{\bar{1}}$.
We have $(\ad a)^2=\ad(a^2)$, so it is sufficient to establish nilpotence of $\ad(a^2)$,
where $a^2=\frac{1}{2}[a,a]\in\RR_{\bar{0}}$
(in case $\ch K=2$, we use the formal square mapping $\RR_{\bar 1}\ni a\mapsto a^2\in\RR_{\bar{0}}$).
Thus, the proof is reduced to the case of an even element $a\in\RR_{\bar{0}}$.
Observe that even standard monomials contain at least one Grassmann variable,
because the pivot elements $v_n$ are odd.

We prove a more general statement. Fix a number $N>0$.
Let $V$ be a set of standard monomials of length at most $N$, each containing at least one Grassmann letter.
We are going to prove that the associative algebra $\Alg(V)\subset\End\Lambda$ is nilpotent.

Let $w$ be a nonzero product of $M$ elements $w_i\in V$.
We get $\wt(w)\geq M$, because $\wt(w_i)\geq 1$.
We transform the product as follows.
We move to the left and order Grassmann variables, while not changing an order of the heads.
Observe that a Grassmann variable $x_j$ can either disappear by commuting
with an appropriated head $v_k$, $k\leq j$,
or be substituted by a product of smaller Grassmann letters.
The following examples show the commutation patterns:
\begin{equation}\label{transformacoes}
\begin{split}
x_2v_5\cdot x_5v_7 & =  -x_2x_5\cdot v_5v_7+x_2\cdot v_7;\\
x_3v_4\cdot x_5v_7 & =  -x_3x_5\cdot v_4v_7;\\
x_2v_3\cdot x_5v_7 & =  -x_2x_5\cdot v_3v_7+x_2x_3x_4\cdot v_7.
\end{split}
\end{equation}

As a result, the product can be written as a linear combination of monomials of the form
$$
u=x_{i_1}\cdots x_{i_{n}}\cdot v_{k_1}\cdots v_{k_{m}},
\qquad i_1<\cdots< i_n<N,\quad k_l\leq N.
$$

By notations above, consider that $n$ is the number of Grassmann variables and
$m$ the number of heads in the product above.
Since $x_{i_l}\in\lbrace x_0,\ldots,x_{N-1}\rbrace$ and the product $u$ is nonzero, we have $n\le N$.
On the other hand, the number of Grassmann variables in the original product $w$
is greater or equal to the number of the heads,
and this property is kept by all transformations~\eqref{transformacoes},
hence $m\leq n\leq N$. We evaluate
weight of the resulting monomial $u$:
$$M\leq\wt(w)=\wt(u)\leq m\cdot\wt(v_N)= m\cdot 2^N \leq N\cdot2^N.$$
Therefore, $\Alg(V)^{N_1}=0$, where $N_1=N\cdot2^N+1$.

Now, consider any even element $a\in\RR_{\bar{0}}$.
Then $a$ is a finite linear combination of even standard monomials,
which contain at least one Grassmann variable, let $N$ be the maximum of their lengths.
By the arguments above, $a^{N_1}=0$.
Set $k=2N_1-1$.
Let $l_a,r_a$ be operators of left and right multiplications by $a$ in $\AA$, respectively.
For any $b\in\RR$, we have
\begin{equation*}
(\ad a)^{k}(b)  =(l_a-r_a)^{k}(b)=\sum_{i+j=k}
\binom{k}i (-1)^jl_a^ir_a^j(b)
=  \sum_{i+j=k} \binom{k}i (-1)^ja^iba^j=0.
\qedhere
\end{equation*}
\end{proof}
\begin{Remark}
We do not know whether $\ad a$  is nilpotent where $a\in\RR$ is a non-$\Z_2$-homogeneous element.
\end{Remark}

To compare with the case of Lie algebras (Theorem~\ref{TMarZel}) we need a weaker version on the $\Z^2$-gradation.
\begin{Corollary}\label{adnilpotentesmultigrau}
Consider the multidegree $\mathbb{Z}^2$-gradation $\RR=\mathop{\oplus}\limits_{n_1,n_2\geq 0}\RR_{n_1 n_2}$
(Lemma~\ref{Lz2graduacao}).
For any $a\in\RR_{n_1n_2}$, $n_1,n_2\geq 0$, the operator $\ad(a)$ is nilpotent.
\end{Corollary}
\begin{proof}
One has $\RR_{n_1,n_2}\subset \RR_{\overline{n_1+n_2}}$ for all $n_1,n_2\geq 0$.
\end{proof}

\begin{Theorem}\label{Tlocal}
Consider the triangular decompositions of Corollary~\ref{Ctriang}.
\begin{enumerate}
\item in $\AA=\AA_+\oplus\AA_0\oplus\AA_-$, the components $\AA_+$, $\AA_-$ are locally nilpotent;
\item in $\RR=\RR_+\oplus\RR_0\oplus\RR_-$, all three components are locally nilpotent.
\end{enumerate}
\end{Theorem}
\begin{proof}
The homogeneous components of $\RR$ and $\AA$ are in regions of plane bounded by pairs of logarithmic curves
(Theorem~\ref{TcurvesL} and Theorem~\ref{TcurvesA}, see also Fig.~\ref{Fig1}).
In case of the positive and negative components,
the result easily follows by the geometric arguments of~\cite[Corollary~5.2]{PeSh09}.

Consider the subalgebra $\RR_0$, it is spanned by non-pivot standard monomials,
i.e. they contain Grassmann variables. Now the arguments of Theorem~\ref{Tadnilpotente}
yield that $\Alg(\RR_0)$ is locally nilpotent.
\end{proof}

\begin{Lemma}\label{Lnaonil}
Let $\ch K=0$ or $\ch K=2$. Then
\begin{enumerate}
\item The associative algebra $\AA=\Alg(v_0,v_1)$ is not nil.
\item Let $\ch K=2$. The restricted Lie algebra $\RR=\Lie_p(v_0,v_1)$ is not nil.
\end{enumerate}
\end{Lemma}
\begin{proof}
Let $\ch K=2$.
Recall that the Lie superalgebra $\RR=\Lie(v_1,v_2)$ coincides with the restricted Lie algebra $\Lie_p(v_0,v_1)$
generated by $v_0,v_1$ (Corollary~\ref{Cchar2}).
So, for any $a\in \RR$ one has the square belonging to~$\RR$.
Fix $n\geq0$ and $0\neq\alpha\in K$.
Consider $v=v_n+\alpha x_n v_{n+1}\in\RR$. We have
\begin{align*}
v^2 & =  (v_n+\alpha x_nv_{n+1})^2 = v_n^2+[v_n,\alpha x_nv_{n+1}]\\
& =  x_{n+1}v_{n+2}+\alpha v_{n+1}
  =  \alpha(v_{n+1}+1/\alpha\cdot x_{n+1}v_{n+2});\\
(v^2)^2&=\alpha^2\cdot 1/{\alpha}(v_{n+2}+\alpha x_{n+2}v_{n+3})
=\alpha(v_{n+2}+\alpha x_{n+2}v_{n+3}).
\end{align*}
By induction, we get
$$
v^{2^k}=\alpha^{r(k)} (v_{n+k}+\alpha^{(-1)^{k}}x_{n+k}v_{n+k+1})\ne 0,\quad\text{where}\ r(k)\in\N,\qquad k\ge 0.
$$
We conclude that $\RR$ is not nil as a restricted Lie algebra.
Hence, its associative hull $\AA$ is not nil as well.

Now, let $\ch K=0$, and $\F_2$ the field with two elements.
A $\Z$-span of pure Lie monomials is a subring
$\WW_{\Z,\mathrm{fin}}(\Lambda_I)\subset \WW_{\mathrm{fin}}(\Lambda_I)$
(see notation in Section~\ref{Sweight})
and infinite sums~\eqref{WW} with coefficients in $\Z$ as well
$\WW_{\Z}(\Lambda_I)\subset \WW(\Lambda_I)$.
We use the ordinary square, which is trivial on pure Lie monomials.
Reduction modulo~2 yields natural epimorphisms:
\begin{align*}
\psi: \WW_{\Z,\mathrm{fin}}(\Lambda_I)
&\twoheadrightarrow \WW_{\Z,\mathrm{fin}}(\Lambda_I)\otimes_{\mathbb{Z}}\mathbb{F}_2\cong
\WW_{\F_2,\mathrm{fin}}(\Lambda_I);\\
\psi: \WW_{\Z}(\Lambda_I)
&\twoheadrightarrow \WW_{\Z}(\Lambda_I)\otimes_{\mathbb{Z}}\mathbb{F}_2\cong
\WW_{\F_2}(\Lambda_I).
\end{align*}
Computations of Theorem~\ref{Tbase} show that we have a $\Z$-subring
$\Lie_{\mathbb{Z}}(v_0,v_1)\subset \Lie_K(v_0,v_1)=\RR\subset\WW(\Lambda)$, consider also
$\Alg_{\mathbb{Z}}(v_0,v_1)\subset \Alg_K(v_0,v_1)=\AA$.
Using Corollary~\ref{C_Lie_subring} and Corollary~\ref{Cchar2},
reduction modulo 2 yields epimorphisms of Lie super-rings and associative rings:
\begin{align*}
\psi:\Lie_{\mathbb{Z}}(v_0,v_1)
 &\twoheadrightarrow
\Lie_{\mathbb{Z}}(v_0,v_1)\otimes_{\mathbb{Z}}\mathbb{F}_2\cong\Lie_{\mathbb{F}_2}(v_0,v_1);\\
\bar\psi:\Alg_{\mathbb{Z}}(v_0,v_1)
 &\twoheadrightarrow\Alg_{\mathbb{Z}}(v_0,v_1)\otimes_{\mathbb{Z}}\mathbb{F}_2\cong\Alg_{\mathbb{F}_2}(v_0,v_1).
\end{align*}
By the first claim, $\Alg_{\mathbb{F}_2}(v_0,v_1)$ is not nil, hence $\AA$ is not nil as well.
\end{proof}

\section{Lie superalgebra $\RR$ is just infinite}\label{secaojustinfinite}

A $K$-algebra $A$ is called {\it just infinite dimensional}, or simply {\it just infinite},
if $\dim_K A=\infty$ and each nonzero ideal of $A$ has finite codimension.
It is also called {\it hereditary just infinite} provided that
any ideal of finite codimension is just infinite.
Recall that by Lemma~\ref{Lpivot}, $\RR=\Lie(v_0,v_1)$ is infinite dimensional.

\begin{Theorem}\label{Tjustinf}
The Lie superalgebra $\RR=\Lie(v_0,v_1)$ is just infinite.
\end{Theorem}

\begin{proof}
Assume that $J\triangleleft\RR$ is nonzero ideal. Let $0\neq a\in J$.

A) By Theorem~\ref{Tbase} (Corollary~\ref{Cchar2} in case $\ch K=2$), $a$ is
a finite linear combination of standard monomials.
Among these monomials,
consider those of greater weight, which are at most two: $w_1=r_{N-2}v_N$ and $w_2=r_{N-2}x_Nv_{N+1}$
(Lemma~\ref{Ldoismonomios}).
Thus,
\begin{equation*}
a=\tilde{a}+\alpha r_{N-2}v_N+\beta r_{N-2}x_Nv_{N+1},\quad \a,\b\in K;
\end{equation*}
where at least one of $\lbrace\alpha,\beta\rbrace$ is nonzero
and $\tilde{a}$ contains standard monomials which weights
are less than $\wt(r_{N-2}v_N)=\wt(r_{N-2}x_Nv_{N+1})$.
We multiply $a$ by all pivot elements $v_i$ such that $x_i$
is a factor of $r_{N-2}$, thus deleting the whole tail $r_{N-2}$.
As a result, we get
\begin{align}
\label{vw}
a'&=\tilde a'+\alpha v_N+\beta x_Nv_{N+1}\in J;\\
1     &\le \wt(\tilde a')<\wt(v_N)=2^N,
\label{vw2}
\end{align}
where the inequalities above stand for all standard monomials in decomposition of $\tilde a'$.

B) Assume that $\alpha=0$, then $\beta\ne 0$. In this case we take
$$a''=[v_n,a']=[v_n,\tilde a']+\beta v_{N+1}\in J,\quad \b\ne 0.$$
Thus, without loss of generality, we can consider that $\a\ne 0$ in~\eqref{vw}.
Recall that
$[v_{N-1}^2,v_N]=-v_{N+1}$,
$[v_{N-1}^2,x_Nv_{N+1}]=[x_Nv_{N+1},x_Nv_{N+1}]=0$ (see Lemma~\ref{Lbasicas}).
Also, we use~\eqref{vw} and estimates~\eqref{vw2}:
\begin{align}
\label{vwww2}
b  &=[v_{N-1}^2,a']=b'-\a v_{N+1}\in J,\quad \text{where}\\
b' &=[v_{N-1}^2,\tilde a'];\nonumber\\
2^N<2\wt(v_{N-1})+1 &\le \wt(b')< 2\wt(v_{N-1})+2^N= 2^{N+1}\label{upper}.
\end{align}

Comparing~\eqref{upper} with bounds of Lemma~\ref{Lestimativas},
we conclude that $b'$ can only have
monomials of the first type of length $N+1$ and ones of the second type of length $N+2$.
A scalar multiple of~\eqref{vwww2} has a form:
\begin{equation}
\tilde b=\sum_{i}\lambda_i r^{(i)}_{N-1}v_{N+1}+\sum_j\mu_j
{\bar r}^{(j)}_{N-1}x_{N+1}v_{N+2}+ v_{N+1}\in J,\qquad \lambda_i,\mu_j\in K.
\end{equation}
Since upper bound~\eqref{upper} is strict, all tails $r^{(i)}_{N-1}$, ${\bar r}^{(j)}_{N-1}$
have at least one Grassmann variable $\{x_0,\ldots,x_{N-1}\}$.

C) Take
\begin{equation*}
[x_0\cdots x_Nv_{N+2},\tilde b]=- x_0\cdots x_{N}x_{N+1}v_{N+3}\in J.
\end{equation*}
Multiplying by $v_0,\ldots,v_{N+1}$, we delete the respective Grassmann letters and obtain $v_{N+3}\in J$.

D) By Lemma~\ref{Lbasicas},
$[v_{N+2}^2,v_{N+3}]=-v_{N+4}\in J$. Similarly, by induction, $v_n\in J$ for all $n\geq N+3$.
Let us prove that any standard monomial of length $n\geq N+6$ belongs to $J$.
Let $n\geq N+5$. Since $v_{n-2}\in J$ we have
$$[x_0\cdots x_{n-3}v_{n-1},v_{n-2}]=- x_0\cdots x_{n-2}v_n\in J.$$
Multiplying by $v_i$, $i=0,\ldots,n-2$, we can delete any letters $x_i$   above.
Hence, $J$ contains all standard monomials of the first type of length $n\geq N+5$.

Consider the case $\ch K\neq 2$. Let $n\geq N+5$ be even.
Since $v_{n-2}\in J$, using Lemma~\ref{Lprodutos}, we have
$$[v_0,v_{n-2}]=2x_0\cdots x_{n-3}x_{n-1}v_n\in J.$$
Similarly, we can delete any letters $x_i$, $i=0,\ldots,n-3$, above. Let $n\ge N+6$ be odd.
Since $v_{n-3}\in J$, using Lemma~\ref{Lbasicas}, we get
\begin{align*}
[x_0 v_{n-4},[v_{n-4},v_{n-3}]]&=[x_0 v_{n-4},-x_{n-4}v_{n-2}]=-x_0 v_{n-2}\in J;\\
[v_1,x_0 v_{n-2}]&=2x_0\cdots x_{n-3}x_{n-1}v_n\in J.
\end{align*}
Similarly, we can delete any letters $x_i$, $i=0,\ldots,n-3$, above.
Hence, all standard monomials of the second type of length $n\geq N+6$ belong to $J$.

Consider $\ch K=2$. Using the formal quadratic mapping,
$v_{n-2}^2=x_{n-1}v_{n}\in J$ for all $n\geq N+5$.

We proved that $J$ contains all basis monomials of $\RR$ of length greater or equal to $N+6$.
Therefore, dimension of the quotient algebra $\RR/J$ is finite,
bounded by number of the basis monomials of $\RR$ of length less than $N+6$.
\end{proof}

\begin{Lemma}\label{Lnonjustinf}
The Lie superalgebra $\RR=\Lie(v_0,v_1)$ is not hereditary just infinite.
\end{Lemma}
\begin{proof}
Fix $m\ge 1$.
Let $\RR(m)\subset\RR$ be the linear span of all basis monomials of $\RR$ of length at least $m$.
By multiplication rules (see Lemma~\ref{Lprodutos} and \eqref{action}), $\RR(m)$ is an ideal of $\RR$.
Its codimension is equal to a finite number of the basis  monomials of length less than $m$.
(In particular, $\dim\RR/\RR(1)=1$).
Let $J=x_0\RR(m)$ be the subspace of $\RR(m)$ spanned by its basis monomials containing $x_0$.
By multiplication rules, $J$ is an abelian ideal of $\RR(m)$.
Since $v_i\in \RR(m){\setminus} J$ for all $i\ge m$, we conclude that
$\dim \RR(m)/J=\infty$.
\end{proof}
\begin{Remark}
We conjecture that the 3-generated Lie superalgebra $\QQ$ of~\cite{Pe16} is also just infinite but not hereditary just infinite.
\end{Remark}
\subsection*{Acknowledgments}
The authors are grateful to Alexei Krasilnikov, Plamen Koshlukov, Ivan Shestakov, Said Sidki, Dessislava Kochloukova,
and Dmitry Millionschikov for useful discussions.

\end{document}